\documentclass[a4paper,reqno]{amsart}
\usepackage{amssymb}
\usepackage[mathcal]{euscript}
\usepackage[cmtip,all]{xy}

\usepackage{tikz-cd}

\newcommand{\bydef}{:=}

\newcommand{\wh}[1]{\widehat{#1}}

\newcommand{\id}{\mathrm{id}}

\newcommand{\op}{\mathrm{op}}
\DeclareMathOperator{\Int}{\mathrm{Int}} 

\newcommand{\cA}{\mathcal{A}}
\newcommand{\cB}{\mathcal{B}}
\newcommand{\cC}{\mathcal{C}}
\newcommand{\cD}{\mathcal{D}}

\newcommand{\cH}{\mathcal{H}}

\newcommand{\cL}{\mathcal{L}}

\newcommand{\cO}{\mathcal{O}}

\newcommand{\NN}{\mathbb{N}}
\newcommand{\ZZ}{\mathbb{Z}}

\newcommand{\RR}{\mathbb{R}}
\newcommand{\CC}{\mathbb{C}}

\newcommand{\FF}{\mathbb{F}}
\newcommand{\KK}{\mathbb{K}}

\DeclareMathOperator{\Hom}{\mathrm{Hom}}
\DeclareMathOperator{\End}{\mathrm{End}}

\DeclareMathOperator{\Aut}{\mathrm{Aut}}

\DeclareMathOperator{\Maps}{\mathrm{Maps}}
\DeclareMathOperator{\Cent}{\mathrm{Cent}}
\DeclareMathOperator{\Ind}{\mathrm{Ind}}
\DeclareMathOperator{\uE}{\mathrm{E}}
\DeclareMathOperator{\res}{\mathrm{res}}
\DeclareMathOperator{\Mat}{\mathrm{Mat}}
\DeclareMathOperator{\alg}{\mathrm{alg}}

\newenvironment{romanenumerate}
 {\begin{enumerate}
 
 }{\end{enumerate}}

\newcommand{\cE}{\mathcal{E}}

\newcommand{\LL}{\mathbb{L}}

\newcommand{\Gal}{\mathrm{Gal}} 
\newcommand{\Br}{\mathrm{Br}} 

\DeclareMathOperator{\Ext}{\mathrm{Ext}} 

\newcommand{\Hc}{\textup{H}}  
\newcommand{\Zc}{\textup{Z}}
\newcommand{\Bc}{\textup{B}}
\newcommand{\dc}{\textup{d}}

\newtheorem{theorem}{Theorem}[section]
\newtheorem{proposition}[theorem]{Proposition}
\newtheorem{lemma}[theorem]{Lemma}
\newtheorem{corollary}[theorem]{Corollary}

\theoremstyle{definition}
\newtheorem{df}[theorem]{Definition}

\newtheorem{examples}[theorem]{Examples}

\theoremstyle{remark}
\newtheorem{remark}[theorem]{Remark}

\numberwithin{equation}{section}

\begin{document}

\title{Graded-division algebras and Galois extensions}

\author{Alberto Elduque${}^\star$}
\address{Departamento de Matem\'{a}ticas
 e Instituto Universitario de Matem\'aticas y Aplicaciones,
 Universidad de Zaragoza, 50009 Zaragoza, Spain}
\email{elduque@unizar.es}
\thanks{${}^\star$Supported by grants MTM2017-83506-C2-1-P (AEI/FEDER, UE) and E22\_17R 
(Gobierno de Arag\'on, Grupo de referencia ``\'Algebra y
Geometr{\'\i}a'', cofunded by Feder 2014-2020 ``Construyendo Europa desde Arag\'on'')}

\author{Mikhail Kochetov${}^{\star\star}$}
\address{Department of Mathematics and Statistics,
 Memorial University of Newfoundland,
 St. John's, NL, A1C5S7, Canada}
\email{mikhail@mun.ca}
\thanks{${}^{\star\star}$Supported by Discovery Grant 2018-04883 of the Natural Sciences and Engineering Research Council (NSERC) of Canada.}




\begin{abstract}
Graded-division algebras are building blocks in the theory of finite-dimensional
associative algebras graded by a group $G$. 
If $G$ is abelian, they can be described, using a loop construction, in terms of central
simple graded-division algebras.

On the other hand, given a finite abelian group $G$, any central simple $G$-graded-division algebra
over a field $\FF$ is determined, thanks to a result of Picco and Platzeck, by
its class in the (ordinary) Brauer group of $\FF$ and the isomorphism class of a $G$-Galois
extension of $\FF$. 

This connection is used to classify the simple $G$-Galois extensions of $\FF$ in terms of a
Galois field extension $\LL/\FF$ with Galois group isomorphic to a quotient $G/K$ and 
an element in the quotient $\Zc^2(K,\LL^\times)/\Bc^2(K,\FF^\times)$ subject to certain
conditions. Non-simple $G$-Galois extensions are induced from simple $T$-Galois extensions 
for a subgroup $T$ of $G$. We also classify finite-dimensional $G$-graded-division algebras 
and, as an application, finite $G$-graded-division rings.
\end{abstract}

\maketitle

\section{Introduction}\label{se:intro}

Division algebras and the Brauer group play a key role in the theory of finite-dimen\-sional associative algebras. 
Any finite-dimensional central simple associative algebra over a field $\FF$ is isomorphic to  
a matrix algebra $\Mat_n(\cD)$ over a central division algebra $\cD$ or, in other words, 
the algebra $\End_\cD(V)$ of endomorphisms of a finite rank right $\cD$-module $V$. 
The Brauer group $\Br(\FF)$ is the group of equivalence classes of finite-dimensional central simple algebras,
with two such algebras being equivalent if they are isomorphic to 
matrix algebras over the same division algebra. The equivalence class of $\cA$, which is an element of $\Br(\FF)$, 
will be denoted by $[\cA]$. The multiplication in $\Br(\FF)$ is induced by the tensor product of $\FF$-algebras: 
$[\cA][\cB]=[\cA\otimes_\FF\cB]$. We are interested in the graded version of this theory.

Given a group $G$, a \emph{$G$-graded algebra} over a field $\FF$ is 
an algebra $\cA$ endowed with a \emph{$G$-grading}, i.e., a vector space decomposition 
$\Gamma:\cA=\bigoplus_{g\in G}\cA_g$, such that $\cA_{g}\cA_h\subset \cA_{gh}$ 
for all $g,h\in G$. The subspaces $\cA_g$ are called the \emph{homogeneous
components} of the grading $\Gamma$ (or of the graded algebra $\cA$), 
and the elements of $\bigcup_{g\in G}\cA_g$ are said to be \emph{homogeneous}.
The nonzero elements of $\cA_g$ are said to have \emph{degree} $g$. 
The \emph{support} of $\Gamma$ (or of $\cA$) is the subset
$\{g\in G\mid \cA_g\neq 0\}$ of $G$. A \emph{homomorphism of $G$-graded algebras} $\cA\to\cB$ 
is an algebra homomorphism $\psi:\cA\rightarrow\cB$ such that $\psi(\cA_g)\subset\cB_g$
for all $g\in G$. In particular, $\cA$ and $\cB$ are said to be \emph{graded-isomorphic} if there is an
isomorphism of graded algebras $\cA\to\cB$. 

\subsection{Graded-division algebras and the graded Brauer group}\label{sse:intro_BrG}
A $G$-graded (associative) algebra is said to be a \emph{graded-division algebra} 
if all nonzero homogeneous elements are invertible. The support is then a subgroup of $G$.

Graded-division algebras are key objects in the graded theory, as any
finite-dimensional $G$-graded-central-simple associative algebra $\cA$, i.e., a $G$-graded associative algebra $\cA$ 
that is \emph{graded-central} ($Z(\cA)\cap\cA_e=\FF 1$, where $e$ is the identity of $G$) 
and \emph{graded-simple} ($\cA$ has no proper graded ideals), is graded-isomorphic to $\End_\cD(V)$ 
where $\cD$ is a $G$-graded-central-division algebra and $V$ is a graded right $\cD$-module of finite rank. 
The reader may consult the book \cite{NvO} or Chapter~2 in our monograph \cite{EKmon}, 
where this theory is used to study gradings on finite-dimensional simple Lie algebras.

However, the definition of a ``graded Brauer group'' is not evident. For abelian $G$
(for example, $G=\ZZ/2$ in the case of the Brauer-Wall group), one possibility is to fix a bicharacter 
$\phi:G\times G\rightarrow \FF^\times$ and define central simple algebras and (twisted) tensor
products relative to $\phi$ (see \cite{Wall,Knus}). 
Far-reaching generalizations are obtained using algebras with an action and coaction of a Hopf algebra 
(see e.g. \cite{CaeBr} and the references therein). But here we will focus on the special case of trivial $\phi$. 
In this setting, the \emph{graded Brauer group} $\Br_G(\FF)$ consists of the equivalence classes of 
finite-dimensional associative algebras that are central simple and $G$-graded, 
with two such algebras $\cA$ and $\cB$ being equivalent if there
is a central simple $G$-graded-division algebra $\cD$ and $G$-graded right 
$\cD$-modules $V$ and $W$ such that $\cA$ is graded-isomorphic to $\End_\cD(V)$ and
$\cB$ to $\End_\cD(W)$. The class of $\cA$ in $\Br_G(\FF)$ will be denoted by $[\cA]_G$. 
The multiplication in $\Br_G(\FF)$ is induced by the standard tensor product: $[\cA]_G[\cB]_G=[\cA\otimes_\FF\cB]_G$, 
where the homogeneous component $(\cA\otimes_\FF\cB)_g$ is defined to be the span of all elements  
$a\otimes b$ with $a\in\cA_{g_1}$, $b\in\cB_{g_2}$, and $g_1 g_2=g$.

Our motivation to consider this setting is explained in the next two subsections, but first we point out that the group 
$\Br_G(\FF)$ depends functorially on both $\FF$ and $G$. 
As with the classical Brauer groups, any embedding of fields $\FF\to\KK$ induces a homomorphism $\Br_G(\FF)\to\Br_G(\KK)$ 
by extension of scalars: the class of $\cA=\bigoplus_{g\in G}\cA_g$ is sent to the class of 
$\cA\otimes_\FF\KK=\bigoplus_{g\in G}\cA_g\otimes_\FF\KK$. 
Also, any group homomorphism $\alpha:G\to H$ induces a homomorphism $\Br_G(\FF)\to\Br_H(\FF)$ by ``push-forward'' 
of grading: the class of $\cA=\bigoplus_{g\in G}\cA_g$ is sent to the class of the same algebra, but equipped with 
the grading $\cA=\bigoplus_{h\in H}\cA_h$ where $\cA_h=\bigoplus_{g\in\alpha^{-1}(h)}\cA_g$.

\subsection{Loop algebra construction}\label{sse:intro_loop}
Given an epimorphism of abelian groups $\pi:G\to\overline{G}$, the ``push-forward'' mentioned above turns any 
$G$-graded algebra to a $\overline{G}$-graded algebra by coarsening the grading. The loop algebra construction is the 
right adjoint of this functor: it sends a $\overline{G}$-graded algebra 
$\cA=\bigoplus_{\bar{g}\in\overline{G}}\cA_{\bar{g}}$ to the $G$-graded algebra 
\[
L_\pi(\cA)\bydef\bigoplus_{g\in G}\cA_{\pi(g)}\otimes g\subset\cA\otimes_\FF \FF G,
\] 
where the multiplication is as in $\cA\otimes_\FF \FF G$ and the $G$-grading is given by the above direct sum.
This construction is well known in Lie theory, but it works for any variety of algebras: since the group algebra $\FF G$ is 
commutative, the loop construction preserves (homogeneous) polynomial identities.

It is shown in \cite{EldLoop}, based on previous results in \cite{BSZ,Allison_et_al,BK}, that, for abelian $G$, 
any $G$-graded-central-simple algebra (not necessarily associative or finite-dimensional) 
is graded-isomorphic to a cocycle-twisted loop algebra of a central simple $\overline{G}$-graded algebra, 
for a suitable quotient $\overline{G}$ of $G$. Thus, all graded-central-simple algebras can be obtained, at least in principle, from gradings on central simple algebras. It is important to point out 
that graded-simple algebras may be far from being simple, or even semisimple,
as ungraded algebras. 

\subsection{Gradings on Lie algebras and their representations}
Let $\cL$ be a finite-dimensional semisimple Lie algebra over a field $\FF$ of 
characteristic $0$ and let $V=V(\lambda)$ be a finite-dimensional irreducible representation
of $\cL_{\overline{\FF}}=\cL\otimes_\FF\overline{\FF}$, where $\overline{\FF}$ is an
algebraic closure of $\FF$ and $\lambda$ is the highest weight of $V$ relative to a Cartan
subalgebra of $\cL_{\overline{\FF}}$. A natural question is whether or not $V$ descends to a representation of $\cL$, 
i.e., whether or not there exists a representation of $\cL$ that becomes 
(isomorphic to) $V$ after extension of scalars to $\overline{\FF}$. 
A necessary condition is that $\lambda$ be invariant under the $*$-action
of the absolute Galois group $\Gal(\overline{\FF}/\FF)$, as in \cite[\S 27.A]{KMRT}.
If this is the case, one can define the \emph{Tits algebra} $A_\lambda$ over $\FF$, with $A_\lambda\otimes_\FF
\overline{\FF}\simeq\End_{\overline{\FF}}(V)$, and a surjective homomorphism from
the universal enveloping algebra $U(\cL)$ onto $A_\lambda$ that, after scalar extension, 
becomes the representation $U(\cL_{\overline{\FF}})\rightarrow\End_{\overline{\FF}}(V)$. 
Then $V$ descends to a representation of $\cL$ if and only if the class of $A_\lambda$
in the Brauer group $\Br(\FF)$ is trivial. 

If $\cL_{\overline{\FF}}$ is graded by an abelian group $G$, a natural question is whether or not $V=V(\lambda)$ 
admits a $G$-grading that makes it a graded $\cL_{\overline{\FF}}$-module. A necessary condition is that $\lambda$ 
be invariant under the action of the dual group $\wh{G}=\Hom(G,\overline{\FF}^\times)$. 
If this is the case, one can define a $G$-grading on $\End_{\overline{\FF}}(V)$ such that 
the representation $U(\cL_{\overline{\FF}})\rightarrow\End_{\overline{\FF}}(V)$ is a homomorphism of graded algebras
--- see \cite{EK_Israel} and \cite[Appendix]{EK15}. Then $V$ admits a $G$-grading if and only if the class of 
$\End_{\overline{\FF}}(V)$ in the graded Brauer group $\Br_G(\overline{\FF})$ is trivial.

Now suppose $\cL$ (and hence $\cL_{\overline{\FF}}$) is graded by $G$, $V=V(\lambda)$ admits a $G$-grading, 
and $\lambda$ is $\Gal(\overline{\FF}/\FF)$-invariant. Then we get a $G$-grading on $A_\lambda$ such that 
$U(\cL)\rightarrow A_\lambda$ is a homomorphism of graded algebras, and $V$ descends to a graded $\cL$-module
if and only if the class of $A_\lambda$ in the graded Brauer group $\Br_G(\FF)$ is trivial. 

\subsection{Classifications of graded-division algebras}
There are two natural ways to classify graded-division algebras: up to isomorphism of graded algebras or up to equivalence of graded algebras. If $\cD$ and $\cD'$ are graded-division algebras with supports $T$ and $T'$, then $\cD$ and $\cD'$ are \emph{equivalent} if there exists an isomorphism of algebras $\cD\to\cD'$ that maps $\cD_t$ to $\cD'_{\alpha(t)}$ where $\alpha:T\to T'$ is a group isomorphism. 

If $\cD$ is a graded-division algebra, then the identity component $\cD_e$ is a division algebra, the support $T$ is a subgroup of $G$, and $\cD$ is graded-isomorphic to the crossed product of $\cD_e$ and $T$, for a suitable action and $2$-cocycle in $\Zc^2(T,\cD_e^\times)$, where $\cD_e^\times$ is the group of invertible elements of $\cD_e$. In principle, group actions and cohomology can be used to classify the graded-division algebras with fixed $\cD_e$ and $T$. As shown in \cite{Kar}, they are classified, up to graded-isomorphism, by the following data: (i) a homomorphism $\sigma:T\to\mathrm{Out}(\cD_e)\subset\mathrm{Out}(\cD_e^\times)$ such that the corresponding obstruction in $\Hc^3(T,Z(\cD_e^\times))$ vanishes, and (ii) with $\sigma$ already fixed (up to conjugation in $\mathrm{Out}(\cD_e)$), an orbit in $\Hc^2(T,Z(\cD_e^\times))$ under a certain twisted action of $\Aut(\cD_e,\sigma)$. In practice, however, even if all these homomorphisms and orbits can be found, it is still difficult to construct the corresponding graded-division algebras explicitly and determine their properties.

The situation is more manageable if $\cD_e$ is the ground field $\FF$, which is the case if $\FF$ is algebraically closed and $\cD$ is finite-dimensional. Then $\cD$ is graded-isomorphic to the twisted group algebra $\FF^\tau T$ for some $2$-cocycle $\tau:T\times T\to \FF^\times$ (with $T$ acting trivially on $\FF^\times$), and $\FF^{\tau} T$ is graded-isomorphic to $\FF^{\tau'} T$ if and only if $[\tau]=[\tau']$ in $\Hc^2(T,\FF^\times)$. In particular, if $G$ is abelian and $\FF$ is algebraically closed then finite-dimensional graded-division algebras are classified by pairs $(T,\beta)$ where $T$ is a finite subgroup of $G$ and $\beta:T\times T\to\FF^\times$ is an alternating bicharacter. Moreover, the graded-division algebra can be constructed from this data explicitly; it is simple as an ungraded algebra if and only if $\beta$ is nondegenerate (see \cite{BSZ,BZ} and \cite[Chapter 2]{EKmon}). 

Over an arbitrary field $\FF$, the study of finite-dimensional graded-division algebras can be, in principle, reduced to the case of twisted group algebras by means of extension of scalars to the field $\LL=Z(\cD_e)$, which is a finite Galois  extension of $\FF$, and then using Galois descent (see \cite{BEK}). Over the field of real numbers, an explicit classification up to isomorphism of finite-dimensional graded-division algebras with an abelian grading group was given in \cite{BEK} (see also \cite{ARE} for the simple case). An explicit classification up to equivalence was given for these algebras in \cite{BZreal} (see also \cite{BZreal_simple,ARE} for the simple case).

\subsection{Exact sequence of Picco-Platzeck}\label{sse:intro_PP}
Given a finite abelian group $G$, Picco and Platzeck proved in \cite{PP} the
existence of the following split short exact sequence of abelian groups (see Theorem \ref{th:PP} in the next section):
\[
1\longrightarrow \Br(\FF)\stackrel{\iota}\longrightarrow \Br_G(\FF)
\stackrel{\zeta}\longrightarrow \uE_G(\FF)\longrightarrow 1
\]
where $\uE_G(\FF)$ is the group of isomorphism classes of \emph{$G$-Galois extensions} of $\FF$ 
(Definition \ref{df:Galois}). This allows us to classify 
finite-dimensional $G$-graded-division algebras in terms of (ungraded) division 
algebras and $G$-Galois extensions, but it can also be used to understand the 
structure of these Galois extensions. We will use both directions in this paper.

We note that the splitting of the above sequence is canonical: 
the embedding $\iota$ sends $[\cA]\in\Br(\FF)$ to $[\cA_0]_G\in\Br_G(\FF)$, where $\cA_0$ denotes the algebra $\cA$
equipped with the trivial grading (i.e., $\cA=\cA_e$), and the ``forgetful'' map $\Br_G(\FF)\to\Br(\FF)$,
sending the class $[\cA]_G$ of a central simple $G$-graded algebra $\cA$ 
to the class $[\cA]$ of $\cA$ as an ungraded algebra, is a left inverse of $\iota$.
(If we regard $\Br(\FF)$ as $\Br_1(\FF)$, these maps are induced by the group homomorphisms $1\to G$ and $G\to 1$.)
In particular, $\zeta$ gives a bijection between the isomorphism classes of \emph{division $G$-gradings} 
on matrix algebras over $\FF$ (i.e., the $G$-gradings that turn the said algebras into graded-division algebras)
and the isomorphism classes of $G$-Galois extensions of $\FF$.

\smallskip

The paper is organized as follows.
In Section \ref{se:GEGBG}, we review the definition of $G$-Galois extensions and the exact sequence of Picco-Platzeck. 
In Section \ref{se:div_to_Gal}, we will show that, for a finite abelian group $G$, 
if $\cD$ is a finite-dimensional central simple $G$-graded-division algebra with support $T$, then the
centralizer $\cC=\Cent_\cD(\cD_e)$ of the identity component $\cD_e$ is a 
simple $T$-Galois extension, and the opposite algebra of $\Ind_T^G(\cC)$
is a $G$-Galois extension representing the image of $[\cD]_G$ under $\zeta$ (Theorem~\ref{th:GTCD}).

The surjectivity of $\zeta$ in the sequence of Picco-Platzek shows 
that any $G$-Galois extension is, up to isomorphism, of the form 
$\Ind_T^G(\cC)$ above. This is used in Section~\ref{se:SimpleGal} to describe the structure of simple $G$-Galois
extensions (Theorem \ref{th:main}). They are determined by a Galois field extension
$\LL/\FF$ with Galois group isomorphic to a quotient $G/K$ and an element 
$\xi\in\Zc^2(K,\LL^\times)/\Bc^2(K,\FF^\times)$, where $K$ and $\xi$ satisfy certain conditions (Corollary \ref{co:main}).
In particular, we give an easy-to-use criterion (Corollary \ref{co:main2}) to check if 
a finite-dimensional $\FF$-algebra $\cA$, endowed with an action by automorphisms 
$\sigma:G\rightarrow\Aut_\FF(\cA)$, is a $G$-Galois extension of $\FF$.
We also give an explicit description of the simple Galois extensions in terms of generators and relations  
(Proposition \ref{prop:classification_mu}).

Section \ref{se:gr-div} will make use of the results in Sections \ref{se:div_to_Gal} and \ref{se:SimpleGal},
together with the main results in \cite{EldLoop}, to classify, for any abelian group $G$, all 
finite-dimensional $G$-graded-central-division algebras over a field $\FF$ up to graded-isomorphism 
(Theorem \ref{th:loop}) and Corollary \ref{co:loop}). Specializing to finite fields, we obtain an explicit
classification of finite $G$-graded-division rings (Theorem \ref{th:finite_GDR}). 

\emph{Throughout, all algebras will be assumed unital, associative and finite-dimen\-si\-onal, 
unless stated otherwise.} 
For any $n\in\NN$, we will denote by $G_{[n]}$ and $G^{[n]}$, respectively, the kernel and image 
of the endomorphism $[n]$ of $G$ that sends $g\mapsto g^n$.


\section{Galois extensions and the graded Brauer group}\label{se:GEGBG}

Given an action of a group $G$ on an algebra $\cC$ by automorphisms: $G\rightarrow \Aut_\FF(\cC)$, sending 
$g\in G$ to the automorphism $c\mapsto g\cdot c$ of $\cC$, we will say that $\cC$ is a \emph{$G$-algebra}.
The fixed subalgebra $\{c\in\cC\mid g\cdot c=c\ \forall g\in G\}$ will be denoted by $\cC^G$.
A \emph{homomorphism of $G$-algebras} is a $G$-equivariant homomorphism of algebras, i.e., 
an algebra homomorphism $\psi:\cC_1\rightarrow\cC_2$ such that $\psi(g\cdot c)=g\cdot\psi(c)$ for 
all $g\in G$ and $c\in \cC_1$. 

\begin{df}\label{df:Galois}
Let $G$ be a finite group. A \emph{$G$-Galois extension of $\FF$} is a finite-dimensional unital $G$-algebra $\cC$ 
over $\FF$ such that the action of $G$ on $\cC$ is faithful, $\cC^G=\FF 1$, and the following equivalent 
conditions hold:
\begin{enumerate}
\item[(a)] The homomorphism
\[
\begin{split}
\Phi:\cC\# \FF G&\longrightarrow \End_\FF(\cC)\\
c g\ &\mapsto \ \bigl(x\mapsto c(g\cdot x)\bigr)
\end{split}
\]
is an isomorphism.
\item[(b)] The linear map 
\[
\begin{split}
\cC\otimes_\FF\cC&\longrightarrow \Maps(G,\cC)\\
 a\otimes b\ &\mapsto\, \bigl(g\mapsto a(g\cdot b)\bigr)
\end{split}
\]
is bijective.
\end{enumerate}
\end{df}

The definition of commutative Galois extensions seems to have appeared for the first time in 
\cite{AuslanderGoldman} using condition (a), where $\cC\# \FF G$ denotes the smash product (recalled below), 
which in this case is the same as the skew group ring of $G$ with coefficients in $\cC$. 
Galois extensions that are not necessarily commutative were introduced in \cite[Definition 4.5]{ChaseRosenberg} 
using condition (b), and indicating that it is equivalent to condition (a). 

Note that condition (b) shows immediately that if $\cC$ is a $G$-Galois extension, then so is its 
opposite algebra $\cC^{\mathrm{op}}$, i.e., the algebra with the same underlying vector space as $\cC$, but with multiplication $x.y\bydef yx$.

Galois extensions of an algebraically closed field $\FF$ were classified in \cite{Dav} using a method developed 
(for a different purpose) in \cite{Mov}.

\smallskip

\emph{In what follows, the ground field $\FF$ will be arbitrary, but $G$ will be assumed abelian, 
unless indicated otherwise.} 

\smallskip

We denote by $[\cC]_{\textup{$G$-alg}}$ the isomorphism class of a $G$-algebra $\cC$. 
Let $\uE_G(\FF)$ be the following set:
\[
\uE_G(\FF)=\{[\cC]_{\textup{$G$-alg}}\mid\text{$\cC$ is a $G$-Galois extension of $\FF$}\}.
\]
Given two $G$-Galois extensions of $\FF$, $\cC_1$ and $\cC_2$, the tensor product $\cC_1\otimes_\FF\cC_2$ is 
naturally a $(G\times G)$-Galois extension. Let $H=\{(g,g^{-1})\mid g\in G\}$. Since $G$ is abelian, $H$ is a subgroup of 
$G\times G$, and the fixed subalgebra $(\cC_1\otimes_\FF\cC_2)^H$ is a $G$-Galois extension, using the isomorphism 
$G\simeq (G\times G)/H$, $g\mapsto (g,1)H=(1,g)H$. This defines an abelian group structure on $\uE_G(\FF)$.
The identity element is the class of $(\FF G)^*\simeq\Maps(G,\FF)$, where $G$ acts as follows: 
$(g\cdot f)(h)=f(hg)$ for $g,h\in G$ and $f\in\Maps(G,\FF)$.

\smallskip

Recall the graded Brauer group $\Br_G(\FF)$ and the embedding $\iota:\Br(\FF)\to\Br_G(\FF)$ from 
Subsections \ref{sse:intro_BrG} and \ref{sse:intro_PP}. 
We are now going to define a homomorphism $\zeta:\Br_G(\FF)\to\uE_G(\FF)$ 
to complete the short exact sequence of Picco-Platzeck.
Although it is not stated like this in \cite{PP}, $\zeta$ and its right inverse
$\vartheta:\uE_G(\FF)\rightarrow\Br_G(\FF)$ are given in terms of smash products, which
we briefly recall (see e.g. \cite[Chapter VII]{Sweedler} or \cite[Chapter 4]{Mont}). 

Consider a unital algebra $\cA$ and a Hopf algebra $\cH$ over $\FF$ (which are not necessarily finite-dimensional), 
and suppose $\cA$ is an \emph{$\cH$-module algebra} via a linear map $\cH\otimes_\FF\cA\rightarrow \cA$, $h\otimes a\mapsto h\cdot a$, which means that $h_1\cdot(h_2\cdot a)=(h_1h_2)\cdot a$ and $1\cdot a=a$ 
for all $h_1,h_2\in\cH$ and $a\in\cA$, and
\[
h\cdot (a_1a_2)=\sum\bigl(h_{(1)}\cdot a_1\bigr)\bigl(h_{(2)}\cdot a_2\bigr)
\text{ and } h\cdot 1=\varepsilon(h)1
\]
for all $h\in\cH$ and $a_1,a_1\in\cA$, where the comultiplication of $\cH$ is $\Delta(h)=
\sum h_{(1)}\otimes h_{(2)}$ (using Sweedler's notation) and the counit is 
$\varepsilon$. Under these conditions, the \emph{smash product} $\cA\# \cH$ is 
the algebra defined on the vector space $\cA\otimes_\FF\cH$ by setting 
\[
(a\otimes h)(b\otimes k)=\sum a\bigl(h_{(1)}\cdot b\bigr)\otimes h_{(2)}k
\]
for all $a,b\in\cA$ and $h,k\in\cH$. Both $\cA\simeq\cA\otimes 1$ and $\cH\simeq 
1\otimes\cH$ are subalgebras of $\cA\# \cH$ and, for simplicity, we will write 
$ah$ for the element $a\otimes h$ in $\cA\# \cH$.

For example, if an arbitrary group $G$ acts on $\cA$ by automorphisms, 
then $\cA$ is a module algebra over the group algebra $\FF G$. The smash product $\cA\#\FF G$ 
consists of the formal sums $\sum_{g\in G}a^g g$, with $a^g\in \cA$ for all $g\in G$ and only a finite number
of $a^g$ being nonzero, and the multiplication is determined by $ga=(g\cdot a)g$ for all
$a\in \cA$ and $g\in G$. Moreover, $\cA\#\FF G$ is naturally $G$-graded
with $\bigl(\cA\#\FF G)_g\bydef \cA g$ for any $g\in G$.

Dually, if $G$ is a finite group, consider the dual Hopf algebra $(\FF G)^*$ of the group algebra $\FF G$. 
Then $(\FF G)^*=\bigoplus_{g\in G}\FF\epsilon_g$, where $\epsilon_g: h\mapsto \delta_{g,h}$ (Kronecker's delta). 
The elements $\epsilon_g$ are orthogonal idempotents, and the comultiplication is given by $\Delta(\epsilon_g)
=\sum_{h\in G}\epsilon_{gh^{-1}}\otimes\epsilon_h$. If $\cA=\bigoplus_{g\in G}\cA_g$ is a unital $G$-graded algebra, 
then $\cA$ is an $(\FF G)^*$-module algebra with the following action: for any element $a=\sum_{g\in G}a_g$, 
with $a_g\in\cA_g$, we set
$
\epsilon_g\cdot a=a_g,
$
i.e., the action of $\epsilon_g$ is the projection onto the homogeneous component of degree $g$. 
The smash product $\cA\#(\FF G)^*$ consists of the formal sums $\sum_{g\in G}a^g\epsilon_g$, 
with $a^g\in \cA$ for all $g\in G$, and the multiplication is determined by
\[
(a\epsilon_g)(b\epsilon_h)=\bigl(ab_{gh^{-1}}\bigr)\epsilon_h
\]
for all $a,b\in\cA$ and $g,h\in G$, where $b=\sum_{k\in G}b_k$, $b_k\in\cA_k$. 
Moreover, $G$ acts on $\cA\#(\FF G)^*$ by automorphisms as follows:
\[
g\cdot (a\epsilon_h)\bydef a\epsilon_{hg^{-1}}
\]
for all $g,h\in G$ and $a\in\cA$.

\smallskip

The following result is valid for all finite abelian groups $G$ and unital commutative rings $\FF$, 
but we will restrict ourselves to the case of fields.

\begin{theorem}[{\cite[Section 1]{PP}}]\label{th:PP}
The mapping $\zeta\bigl([\cA]_G\bigr)=[\Cent_{\cA\# (\FF G)^*}(\cA)]_{\textup{$G$-alg}}$ is a well-defined group homomorphism, and 
\begin{equation}\label{eq:PP}
1\longrightarrow \Br(\FF)\stackrel{\iota}\longrightarrow \Br_G(\FF)\stackrel{\zeta}
\longrightarrow\uE_G(\FF)\longrightarrow 1
\end{equation}
is a split exact sequence. \qed
\end{theorem}

\begin{remark}
The $G$-algebra denoted by $AE$ in \cite{PP} is isomorphic to the smash product
$A\# (\FF G)^*$ by means of $ae_g\mapsto a\epsilon_{g^{-1}}$ for $a\in A$ and $g\in G$.
\end{remark}

For a $G$-graded algebra $\cA$, we will denote $\Gamma(\cA)\bydef \Cent_{\cA\# (\FF G)^*}(\cA)$, 
so that, if $\cA$ is central simple, $\zeta\bigl([\cA]_G\bigr)=[\Gamma(\cA)]_{\textup{$G$-alg}}$
in \eqref{eq:PP}. There is an alternative construction of $\Gamma(\cA)$, which is valid for 
any \emph{strongly graded} $\cA$, but we will restrict ourselves to the following special case. 
Assume that, for every $g\in G$, the homogeneous component $\cA_g$ contains an invertible element, say, 
$u_g$. (In other words, $\cA$ is isomorphic to a crossed product of $\cA_e$ and $G$, as $\cA_g=\cA_e u_g$ for all $g\in G$.)
Consider the graded subalgebra $\cC=\Cent_\cA(\cA_e)$. 
For any $g\in G$, the inner automorphism $\Int u_g: x\mapsto u_g x u_g^{-1}$ preserves $\cA_e$, 
and hence also $\cC$. Moreover, its restriction to $\cC$ does not depend on the choice 
of the invertible element $u_g\in\cA_g$, because any other such element has the form $au_g$, with 
invertible $a\in\cA_e$, and $(\Int a)\vert_\cC=\id_\cC$. Therefore, there is a well-defined group homomorphism
\begin{equation}\label{eq:sigma}
\begin{split}
\sigma: G&\longrightarrow \Aut_\FF(\cC)\\
g\,&\mapsto (\Int u_g)\vert_\cC\quad\text{for any invertible homogeneous $u_g$ of degree $g$.}
\end{split}
\end{equation}
As usual, we will write $g\cdot c$ for the image of $c\in\cC$ under $\sigma_g$.

\begin{lemma}\label{lm:crossed}
Let $G$ be a finite abelian group and $\cA$ be a $G$-graded algebra such that, for every $g\in G$, 
the homogeneous component $\cA_g$ contains an invertible element. 
Then $\Cent_{\cA\#(\FF G)^*}(\cA)$ is antiisomorphic to $\Cent_\cA(\cA_e)$ as a $G$-algebra.
\end{lemma}

\begin{proof}
First we compute $\Cent_{\cA\# (\FF G)^*}(\cA)$. Let $x=\sum_{g\in G}a^g \epsilon_g$, 
where $a^g\in \cA$ for all $g\in G$. For any $b_h\in\cA_h$, we have
\begin{align*}
xb_h&=\Bigl(\sum_{g\in G}a^g \epsilon_g\Bigr)b_h=\sum_{g\in G}(a^gb_h) \epsilon_{h^{-1}g},\\
b_hx&=\sum_{g\in G}b_ha^g \epsilon_g,
\end{align*}
so, equating the coefficients of $\epsilon_g$, we obtain: $x\in\Cent_{\cA\#(\FF G)^*}(\cA)$ if and only if
\begin{equation}\label{eq:ahg}
a^{hg}b_h=b_ha^g\quad\forall g,h\in G,\, b_h\in\cA_h.
\end{equation}
With $h=e$, this equation gives $a^g\in \cC\bydef\Cent_\cA(\cA_e)$. 
With invertible $b_h$ and $g=e$, it gives $a^h=b_h a^e b_h^{-1}=h\cdot a^e$, for any $h\in G$.
Conversely, if $a^g=g\cdot c$ where $c\in\cC$, then \eqref{eq:ahg} holds. 
Indeed, we have $b_h=a u_h$ for some $a\in\cA_e$, so 
$b_h a^g=(h\cdot a^g) b_h=(h\cdot(g\cdot c)) b_h=((hg)\cdot c) b_h=a^{hg} b_h$.
Therefore, we obtain:
\[
\Cent_{\cA\# (\FF G)^*}(\cA)
=\Bigl\{\sum_{g\in G}(g\cdot c)\epsilon_g\mid c\in\cC\Bigr\}.
\]

Define a linear isomorphism
\[
\begin{split}
\psi:\cC&\longrightarrow \Cent_{\cA\# (\FF G)^*}(\cA)\\
c\ &\mapsto\ \sum_{g\in G}(g\cdot c) \epsilon_g.
\end{split}
\]
We claim that $\psi$ is an antiisomorphism of algebras, i.e., $\psi(c)\psi(d)=\psi(dc)$ for all $c,d\in \cC$.
Since $\cC$ is a graded subalgebra of $\cA$, we may assume that $d\in\cA_k$ for some $k\in G$. Then
\[
\begin{split}
\psi(c)\psi(d)&=\Bigl(\sum_{g\in G}(g\cdot c) \epsilon_g\Bigr)\Bigl(\sum_{h\in G}(h\cdot d) \epsilon_h\Bigr)\\
&=\sum_{g,h\in G}(g\cdot c)\delta_{gh^{-1},k}(h\cdot d) \epsilon_h\\
&=\sum_{h\in G}((kh)\cdot c)(h\cdot d)\epsilon_h=\sum_{h\in G}(h\cdot d)(h\cdot c)\epsilon_h\quad\text{because of \eqref{eq:ahg}}\\
&=\sum_{h\in G}(h\cdot (dc))\epsilon_h=\psi(dc),
\end{split}
\]
as claimed.

Finally, the action of $G$ on $\cA\#(\FF G)^*$ is given by
$
h\cdot (a \epsilon_g)=a \epsilon_{gh^{-1}}.
$
For $c\in\cC$ and $h\in G$, we compute:
\[
\psi(h\cdot c)=\sum_{g\in G}((gh)\cdot c) \epsilon_g
=\sum_{g\in G}(g\cdot c) \epsilon_{gh^{-1}}=h\cdot \psi(c),
\]
so $\psi$ is $G$-equivariant.
\end{proof}

\begin{remark}
This result can be considered a special case of \cite[Proposition 3.4]{CGO}, but we included a proof for completeness
and also because showing that the $G$-algebra denoted by $(GA)^A$ in \cite{CGO} is, in the setting of \cite{PP},
antiisomorphic to $\Cent_{A\#(\FF G)^*}(A)$, requires computations similar to the above.
\end{remark}

The fact that $\zeta$ in \eqref{eq:PP} is surjective follows from \cite[Lemma 5]{PP}, but there is a 
minor mistake in its proof. The next result gives the correct statement.   

\begin{lemma}\label{lm:PP5}
Let $G$ be a finite abelian group and $\cC$ be a $G$-Galois extension of $\FF$. Then 
$\Cent_{(\cC\# \FF G)\#(\FF G)^*}(\cC\# \FF G)$ is isomorphic to $\cC$ as a $G$-algebra.
\end{lemma}

\begin{proof}
Denote $\cA=\cC\# \FF G$. Clearly, $\cA$ satisfies the hypothesis of Lemma \ref{lm:crossed} (we can take $u_g=g$),
so $\Cent_{\cA\#(\FF G)^*}(\cA)$ is antiisomorphic to $\Cent_\cA(\cA_e)$ as a $G$-algebra. 
On the other hand, we can identify $\cA$ with $\End_\FF(\cC)$ as in condition (a) of 
Definition \ref{df:Galois}, i.e., an element $c g\in \cA$ is identified with the map $x\mapsto c(g\cdot x)$.
Then the identity component $\cA_e=\cC$ is identified with $L_\cC$, and $\Cent_\cA(\cA_e)$ with  
$\Cent_{\End_\FF(\cC)}(L_\cC)=R_\cC$, where $L_{\cC}$ (respectively, $R_{\cC}$) denotes the subspace of $\End_\FF(\cC)$
spanned by the operators $L_c$ (respectively, $R_c$) of left (respectively, right) multiplication by elements $c\in\cC$. 
Finally, the $G$-action on $\cA$, given by $g\cdot a=gag^{-1}$, corresponds to the natural $G$-action on $\End_\FF(\cC)$:
$(g\cdot f)(x)=g\cdot(f(g^{-1}\cdot x))$ for all $g\in G$, $f\in\End_\FF(\cC)$ and $x\in\cC$. 
One easily checks that $g\cdot R_c=R_{g\cdot c}$ for all $c\in\cC$. 
Therefore, the mapping $c\mapsto R_c$ is an antiisomorphism $\cC\to R_\cC$ as $G$-algebras.
\end{proof}

\begin{remark} In \cite[Lemma 5]{PP} it is incorrectly asserted that 
the centralizer $\Cent_{(\cC\# \FF G)\#(\FF G)^*}(\cC\# \FF G)$ 
is isomorphic to the opposite algebra $\cC^{\textup{op}}$. The problem lies in the fact that 
$u_1\cdot v_1$ is computed instead of $(u\cdot v)_1$ (in the notation in \cite{PP}).
\end{remark}

As a consequence of Lemma \ref{lm:PP5}, the map
\[
\begin{split}
\vartheta: \uE_G(\FF)&\longrightarrow \Br_G(\FF)\\
[\cC]_{\textup{$G$-alg}}&\mapsto [\cC\# \FF G]_G,
\end{split}
\]
is a right inverse of $\zeta$. Note that, since the algebra $\cC\# \FF G$ is isomorphic to 
$\End_\FF(\cC)$, its class in $\Br(\FF)$ is trivial, which means that $\vartheta([\cC]_{\textup{$G$-alg}})$ 
is in the kernel of the ``forgetful'' map $\varphi:\Br_G(\FF)\to\Br(\FF)$.
It follows that $\vartheta$ is an isomorphism $\uE_G(\FF)\simeq \ker\varphi$.

\begin{corollary}\label{co:PP}
Let $G$ be a finite abelian group. Then the map
\[
\begin{split}
\Br_G(\FF)&\longrightarrow \Br(\FF)\times \uE_G(\FF)\\
[\cA]_G\ &\mapsto \left([\cA],[\Gamma(\cA)]_{\textup{$G$-alg}}\right)
\end{split}
\]
is a group isomorphism, and its inverse is the map
\[
\begin{split}
\Br(\FF)\times \uE_G(\FF)&\longrightarrow \Br_G(\FF)\\
\left([\cB],[\cC]_{\textup{$G$-alg}}\right)&\mapsto 
[\cB\otimes_\FF(\cC\# \FF G)]_G,
\end{split}
\]
where the $G$-grading on $\cB\otimes_\FF(\cC\# \FF G)$ is given by
\[
\left(\cB\otimes_\FF(\cC\# \FF G)\right)_g=\cB\otimes_\FF(\cC g).
\]
\end{corollary}


\section{From graded-division algebras to Galois extensions}\label{se:div_to_Gal}

Let $\cD=\bigoplus_{g\in G}\cD_g$ be a central simple graded-division algebra over $\FF$.
Our aim in this section is to express the $G$-Galois extension $\Gamma(\cD)=\Cent_{\cD\# (\FF G)^*}(\cD)$ in simpler terms. 
First we will assume that the support of $\cD$ is the entire $G$, i.e., $\cD_g\ne 0$ for all $g\in G$. 
Then we have a $G$-action on $\cC\bydef\Cent_\cD(\cD_e)$ given by \eqref{eq:sigma}.

\begin{proposition}\label{prop:GCD}
Let $G$ be a finite abelian group and let $\cD$ be a central simple $G$-graded-division algebra with support $G$. 
Then $\Gamma(\cD)$ is a simple algebra, and it is antiisomorphic to $\Cent_\cD(\cD_e)$ as a $G$-algebra. 
\end{proposition}

\begin{proof}
Since $\cD$ satisfies the hypothesis of Lemma \ref{lm:crossed}, $\Gamma(\cD)$ is antiisomorphic to $\cC=\Cent_\cD(\cD_e)$
as a $G$-algebra. The fact that $\cC$ is simple follows from the Double Centralizer Theorem 
(see e.g. \cite[\S 8.5]{Scharlau}), because $\cD$ is central simple, and $\cD_e$ is a division algebra (hence simple).
\end{proof}

Our next step is to suppress the condition on $G$ being the support of $\cD$. 
Denote the support of $\cD$ by $T$ (a subgroup of $G$). Again, let $\cC=\Cent_\cD(\cD_e)$. 
Then we have a $T$-action on $\cC$, and Proposition \ref{prop:GCD} shows that $\cC$ is a $T$-Galois extension of $\FF$. 
Consider the vector space
\[
\Ind_T^G(\cC)\bydef\Hom_{\FF T}(\FF G,\cC),
\]
which can be identified with the algebra of $T$-equivariant maps:
\[
\{f:G\rightarrow \cC\mid f(tg)=t\cdot f(g)\ \forall t\in T\}.
\]
Define pointwise multiplication on $\Ind_T^G(\cC)$: $(f_1 f_2)(g)=f_1(g)f_2(g)$ for all $g\in G$.

As an algebra, $\Ind_T^G(\cC)$ is isomorphic to the Cartesian product of $[G:T]$ copies of $\cC$. 
There is a natural action of $G$ on $\Ind_T^G(\cC)$ by automorphisms, given by
\[
(g\cdot f)(h)=f(hg)
\]
for $g,h\in G$ and $f\in\Ind_T^G(\cC)$.

\begin{theorem}\label{th:GTCD}
Let $G$ be a finite abelian group and let $\cD$ be a central simple $G$-graded-division algebra 
with support $T$. Then $\Gamma(\cD)$ is antiisomorphic to $\Ind_T^G(\cC)$ as a $G$-algebra, 
where $\cC=\Cent_\cD(\cD_e)$ is a simple $T$-Galois extension of $\FF$. 
\end{theorem}

\begin{proof}
Take a transversal $\{g_1=e,g_2,\ldots,g_m\}$ of $T$ in $G$, so that $G$ is the 
disjoint union $G=Tg_1\,\dot\cup\, Tg_2\,\dot\cup\,\cdots\,\dot\cup\, Tg_m$. Then we have
\[
\cD\# (\FF G)^*=\bigoplus_{i=1}^m\Bigl(\bigoplus_{t\in T}\cD\epsilon_{tg_i}\Bigr)
\]
and, for each $i=1,\ldots,m$, $J_i\bydef\bigoplus_{t\in T}\cD\epsilon_{tg_i}$ is an ideal of $\cD\# (\FF G)^*$ 
because $\epsilon_g d_t=d_t\epsilon_{t^{-1}g}$ for all $g\in G$, $t\in T$, $d_t\in\cD_t$, and 
$t^{-1}g$ is in the coset $Tg$. 
Each $J_i$ is naturally isomorphic to $J_1=\bigoplus_{t\in T}\cD \epsilon_t\simeq\cD\# (\FF T)^*$ by means of the map
\[
d\epsilon_{tg_i}\mapsto d\epsilon_t.
\]

As in the proof of Lemma \ref{lm:crossed}, let $x=\sum_{g\in G}a^g\epsilon_g\in\cD\# (\FF G)^*$. Then, by \eqref{eq:ahg}, 
we have $x\in\Gamma(\cD)=\Cent_{\cD\# (\FF G)^*}(\cD)$ if and only if
\[
a^{tg}b_t=b_ta^g\quad\forall g\in G,\,t\in T,\,b_t\in\cD_t.
\]
This forces $a^g\in\cC$ and $a^{tg}=b_t a^g b_t^{-1}$ for any $0\neq b_t\in\cD_t$, i.e., 
$a^{tg}=t\cdot a^g$ for all $g\in G$ and $t\in T$.
Therefore, $\Gamma(\cD)$ is the subalgebra
\[
\Gamma(\cD)=\left\{
\sum_{i=1}^m\Bigl(\sum_{t\in T}(t\cdot a^{g_i})\epsilon_{tg_i}\Bigr)\mid 
a^{g_1},\ldots, a^{g_m}\in\cC
\right\},
\]
which is a direct sum of ideals: $\Gamma(\cD)=\Gamma(\cD)_1\oplus\cdots\oplus\Gamma(\cD)_m$, 
where
\[
\Gamma(\cD)_i\bydef\Bigl\{\sum_{t\in T}(t\cdot c)\epsilon_{tg_i}\mid c\in\cC\Bigr\}=\Gamma(\cD)\cap J_i.
\]
All $\Gamma(\cD)_i$ are isomorphic to $\Gamma(\cD)_1$, which is isomorphic to $\cC^{\textup{op}}$ 
by Proposition \ref{prop:GCD}.

Define an injective linear map
\[
\begin{split}
\Psi: \Ind_T^G(\cC)&\longrightarrow \cD\# (\FF G)^*\\
f&\mapsto \sum_{g\in G}f(g)\epsilon_g.
\end{split}
\]
Note that, for $f\in\Ind_T^G(\cC)$, we have $f(tg_i)=t\cdot f(g_i)$, so
\[
\Psi(f)=\sum_{i=1}^m\Bigl(\sum_{t\in T}\bigl(t\cdot f(g_i)\bigr)\epsilon_{tg_i}\Bigr).
\]
Therefore, the image of $\Psi$ is precisely $\Gamma(\cD)$. For any $h\in G$, we have
\[
\Psi(h\cdot f)=\sum_{g\in G}(h\cdot f)(g)\epsilon_g=\sum_{g\in G}f(gh)\epsilon_g
=h\cdot\Bigl(\sum_{g\in G}f(gh)\epsilon_{gh}\Bigr)=h\cdot\Psi(f),
\]
so $\Psi$ is $G$-equivariant.

Finally, $\Psi$ is the composition of the algebra isomorphism
\[
\begin{split}
\Ind_T^G(\cC)&\longrightarrow \cC\times\stackrel{m}{\cdots} \times\cC\\
f\ &\mapsto \bigl(f(g_1),\ldots,f(g_m)\bigr),
\end{split}
\]
and of the antiisomorphism (Proposition \ref{prop:GCD})
\[
\begin{split}
\cC\times\stackrel{m}{\cdots} \times\cC&\longrightarrow \Gamma(\cD)=\Gamma(\cD)_1\oplus\cdots\oplus\Gamma(\cD)_m\\
(c_1,\ldots,c_m)&\mapsto\Bigl(
\sum_{t\in T}(t\cdot c_1)\epsilon_{tg_1},\ldots,\sum_{t\in T}(t\cdot c_m)\epsilon_{tg_m}\Bigr),
\end{split}
\]
so $\Psi:\Ind_T^G(\cC)\to\Gamma(\cD)$ is an antiisomorphism of $G$-algebras.
\end{proof}

On $\Gamma(\cD)$, there is not only a $G$-action, but also a $G$-grading, coming from the $G$-grading on $\cC$:
$f\in\Ind_T^G(\cC)$ is homogeneous of degree $h\in G$ if $f(g)\in\cC_h$ for all $g\in G$. The grading on $\cC$ 
can be defined intrinsically by 
\begin{equation}\label{eq:MUgr}
\cC_h=\{a\in\cC\mid ab=(h\cdot b)a\ \forall b\in\cC\}.
\end{equation}

\begin{remark}\label{re:Miyashita-Ulbrich}
This is an instance of the so-called Miyashita-Ulbrich action defined in the context of Hopf-Galois extensions
(see e.g. \cite[Chapter 8]{Mont}), in this case of the Hopf algebra $(\FF G)^*$ on a $G$-Galois extension $\cA$ of $\FF$: 
this is the action $\cA\otimes(\FF G)^*\to\cA$, sending $a\otimes\varphi\mapsto a\leftharpoonup\varphi$, 
characterized by the property $ab=\sum_{g\in G}(g\cdot b)(a\leftharpoonup\epsilon_g)$ for all $a,b\in\cA$ 
(cf. \cite[Lemma 3.4]{GN}).
Another instance of Miyashita-Ulbrich action is the $T$-action on $\cC=\Cent_\cD(\cD_e)$, associated to $\cD$ 
as a Hopf-Galois extension of $\cD_e$.
\end{remark}

For any finite abelian groups $H\subset G$, the surjectivity of $\zeta$ in Theorem \ref{th:PP} (applied to $H$), 
together with Theorem \ref{th:GTCD} and the natural isomorphism $\Ind_T^G\simeq\Ind_H^G\circ\Ind_T^H$,
imply that the functor $\Ind_H^G$ sends $H$-Galois extensions to $G$-Galois extensions (cf. \cite[Proposition 3.2]{Dav}). 
We can say more: the group homomorphisms $\zeta=\zeta_G$ in Theorem \ref{th:PP} 
form a natural transformation of functors $\Br_G(\FF)\to\uE_G(\FF)$ with respect to monomorphisms of finite abelian groups.

\begin{remark}\label{re:direct_limit}
This allows us to extend Theorem \ref{th:PP} to arbitrary abelian groups. The definition of $\Br_G(\FF)$ in 
Subsection \ref{sse:intro_BrG} does not require that $G$ be finite, but the support of any finite-dimensional 
$G$-graded-division algebra is a finite subgroup of $G$, hence $\Br_G(\FF)$ can be identified with the direct limit 
of the groups $\Br_H(\FF)$ over all finite subgroups $H$ of $G$ ordered by inclusion. 
Clearly, $\iota=\iota_G$ is then identified with the direct limit of the group homomorphisms $\iota_H$. 
We \emph{define} $\uE_G(\FF)$ to be the direct limit of the groups $\uE_H(\FF)$ and $\zeta_G$ to be the direct limit of  
the group homomorphisms $\zeta_H$ over the same ordered set. Thus we obtain the short exact sequence \eqref{eq:PP}
for an arbitrary abelian group $G$, with a splitting given by the ``forgetful'' map $\Br_G(\FF)\to\Br(\FF)$.
\end{remark}

\smallskip

We will now investigate the relationship between $\cD$ and $\cC=\Cent_\cD(\cD_e)$, and the structure of the latter. 
These results will be used in the next section, but are also of independent interest.

Recall that, given a group $K$, a field $\LL$ and a $2$-cocycle $\tau\in \Zc^2(K,\LL^\times)$ (with trivial 
action of $K$ on $\LL^\times$), the \emph{twisted group algebra} $\LL^\tau K$ is the $\LL$-algebra with 
basis $\{X_k\mid k\in K\}$ and multiplication given by 
\[
X_{k_1}X_{k_2}=\tau(k_1,k_2)X_{k_1k_2}
\]
for any $k_1,k_2\in K$. $\LL^\tau K$ is naturally $K$-graded, and the graded-isomorphism class is 
determined by the class of $\tau$ in the second cohomology group: $[\tau]\in \Hc^2(K,\LL^\times)$. 
Any graded-division algebra over $\LL$ with support $K$ and $1$-dimensional homogeneous components is, 
up to a graded isomorphism, a twisted group algebra $\LL^\tau K$.

We list some facts from \cite{BEK}, which are easy to verify and hold in a more general setting than what we 
have here:

\begin{proposition}[{\cite[Section 2]{BEK}}]\label{pr:DCK}
Let $G$ be a finite group, and let $\cD=\bigoplus_{g\in G}\cD_g$ be a graded-division ring with support $G$. 
Assume that $\cD$ is finite-dimen\-si\-onal over the field $\FF\bydef Z(\cD)\cap\cD_e$.
Denote $\LL\bydef Z(\cD_e)$,
\[
K\bydef \{k\in G\mid \cD_k\cap \Cent_\cD(\LL)\neq 0\},
\]
$\cD_K\bydef\bigoplus_{k\in K}\cD_k$, and $\cC\bydef\Cent_\cD(\cD_e)$. Then the following assertions hold:
\begin{romanenumerate}
\item $K$ is a normal subgroup of $G$ and $\Cent_\cD(\LL)=\cD_K$.

\item The extension $\LL/\FF$ is a Galois field extension, and the mapping
\begin{equation}\label{eq:barsigma}
\begin{split}
\bar\sigma: G&\longrightarrow \Aut_\FF(\LL)=\Gal(\LL/\FF)\\
g\,&\mapsto (\Int u_g)\vert_\LL\quad\text{for any }0\ne u_g\in\cD_g
\end{split}
\end{equation}
is a well-defined surjective group homomorphism with kernel $K$.

\item $\cC$ is a graded subalgebra of $\cD$ with support $K$ and $\cC_e=\LL$, 
hence graded-isomorphic to the twisted group algebra $\LL^\tau K$ for some $\tau\in\Zc^2(K,\LL^\times)$.

\item $\cD_K\simeq \cD_e\otimes_\LL\cC$.\qed
\end{romanenumerate}
\end{proposition}

\begin{corollary}\label{co:DCK}
Under the hypotheses of Proposition \ref{pr:DCK}, assume further that $\cD$ is simple and $Z(\cD)=\FF$. 
Then $\cC$ is simple with $Z(\cC)=\LL$ and the order $\lvert K\rvert$ is a square.
\end{corollary}

\begin{proof}
As in the proof of Proposition \ref{prop:GCD}, we can apply the Double Centralizer Theorem to 
the simple subalgebra $\cD_e$ of the central simple algebra $\cD$, 
so $\cC=\Cent_\cD(\cD_e)$ is simple, and $Z(\cC)=\cC\cap\Cent_\cD(\cC)=\cC\cap\cD_e=\LL$. 
Since $\cC$ is a central simple $\LL$-algebra, $\lvert K\rvert=\dim_\LL\cC$ is a square.
\end{proof}

Recall that the $G$-action on $\cC$ is defined by the group homomorphism $\sigma: G\rightarrow \Aut_\FF(\cC)$, 
$g\mapsto \sigma_g$, given by \eqref{eq:sigma}.
Comparing with \eqref{eq:barsigma}, we see that $\bar{\sigma}_g$ is the restriction of $\sigma_g$ to $\LL$, so 
$\sigma_g$ is a $\bar{\sigma}_g$-semilinear automorphism of $\cC$ as an $\LL$-algebra: 
$\sigma_g(lc)=\bar{\sigma}_g(l)\sigma_g(c)$ for all $l\in\LL$ and $c\in\cC$.
 
\begin{proposition}\label{pr:Kbeta}
Under the hypotheses of Proposition \ref{pr:DCK}, assume further that $G$ is abelian. 
Fixing a nonzero element $X_k$ in each homogeneous component $\cC_k$, we identify $\cC$ with 
$\LL^\tau K=\bigoplus_{k\in K}\LL X_k$. For any $k\in K$ and $g\in G$, define the element $f_k(g)\in\LL^\times$ by
\begin{equation}\label{eq:fks}
\sigma_g(X_k)=f_k(g)X_k.
\end{equation} 
Then we have the following:
\begin{romanenumerate}
\item For any $k\in K$, $f_k:G\rightarrow \LL^\times$ is a $1$-cocycle: $f_k\in\Zc^1(G,\LL^\times)$.

\item Replacing the element $X_k$ by $X_k'=lX_k$, $l\in\LL^\times$, changes $f_k$ to the cohomologous $1$-cocycle 
$f_k'=(\dc l)f_k$, where $\dc l:G\to\LL^\times$ is the $1$-coboundary $g\mapsto \bar\sigma_g(l)l^{-1}$. 
In particular, the class $[f_k]$  of $f_k$ in the cohomology group 
$\Hc^1(G,\LL^\times)=\Zc^1(G,\LL^\times)/\Bc^1(G,\LL^\times)$ does not depend on the choice of $X_k$.

\item The alternating bicharacter $\beta: K\times K\rightarrow \LL^\times$ given by 
\[
\beta(k_1,k_2)=\tau(k_1,k_2)\tau(k_2,k_1)^{-1}
\]
takes values in $\FF^\times$, depends only on the class $[\tau]\in\Hc^2(K,\LL^\times)$, and satisfies
\begin{equation}\label{eq:betaf}
f_k(g)=\beta(g,k)\quad\forall k,g\in K. 
\end{equation}

\item For any $k_1,k_2\in K$, 
\begin{equation}\label{eq:fk1k2}
f_{k_1}f_{k_2}=\dc\bigl(\tau(k_1,k_2)\bigr)f_{k_1k_2}.
\end{equation}

\item The map
\begin{equation}\label{eq:KH1}
\begin{split}
f:K&\longrightarrow \Hc^1(G,\LL^\times)\\
k\,&\mapsto\quad [f_k]
\end{split}
\end{equation}
is a group homomorphism whose kernel is the support of the graded subalgebra $Z(\cD)$. 

\item The following are equivalent: \textup{(a)} $\cD$ is central simple over $\FF$, 
\textup{(b)} $\cC$ is central simple over $\LL$, and \textup{(c)} $\beta$ is nondegenerate.
\end{romanenumerate}
\end{proposition}

\begin{proof}
(i) For any $g_1,g_2\in G$ and $k\in K$, we have:
\[
f_k(g_1g_2)X_k=\sigma_{g_1g_2}(X_k)=\sigma_{g_1}\bigl(\sigma_{g_2}(X_k)\bigr)=\sigma_{g_1}\bigl(f_k(g_2)X_k\bigr)
=\bar\sigma_{g_1}\bigl(f_k(g_2)\bigr)f_k(g_1)X_k,
\]
so $f_k(g_1g_2)=f_k(g_1)\,\bar\sigma_{g_1}\bigl(f_k(g_2)\bigr)$, which means $f_k\in\Zc^1(G,\LL^\times)$.

\medskip

\noindent(ii) For any $g\in G$, we have:
\[
f_k'(g)X_k'=\sigma_g(X_k')=\sigma_g(l X_k)=\bar\sigma_g(l)f_k(g)X_k=\dc l(g)lf_k(g)X_k=\dc l(g)f_k(g)X_k',
\]
so $f_k'=(\dc l)f_k$.

\medskip

\noindent(iii) For any $k_1,k_2\in K$, we have $X_{k_1}X_{k_2}=\beta(k_1,k_2)X_{k_2}X_{k_1}$, 
which implies that $\beta$ does not depend on the choice of the elements $X_k$, $k\in K$, and is an alternating bicharacter. 
In view of \eqref{eq:sigma}, the fact that 
$(\Int X_g)X_k=\beta(g,k)X_k$, for all $g,k\in K$, implies \eqref{eq:betaf}. 
Since $\beta(\cdot,k)$ is the restriction of the $1$-cocycle $f_k$ to the subgroup $K$, it must take values in $\LL^G=\FF$
(see \eqref{eq:res} below). This is also easy to show directly:  
\[
\sigma_g(X_{k_1}X_{k_2})=f_{k_1}(g)f_{k_2}(g)X_{k_1}X_{k_2}=f_{k_1}(g)f_{k_2}(g)\beta(k_1,k_2)X_{k_2}X_{k_1}
\]
must be equal to 
\[
\sigma_g\bigl(\beta(k_1,k_2)X_{k_2}X_{k_1}\bigr)=\bar\sigma_g\bigl(\beta(k_1,k_2)\bigr) f_{k_2}(g)f_{k_1}(g)X_{k_2}X_{k_1},
\]
so $\beta(k_1,k_2)$ is fixed by $\bar\sigma_g$, for any $g\in G$.

\medskip

\noindent(iv) For any $g\in G$ and $k_1,k_2\in K$,
\[
\sigma_g(X_{k_1}X_{k_2})=\sigma_g\bigl(\tau(k_1,k_2)X_{k_1k_2}\bigr)
=\bar\sigma_g\bigl(\tau(k_1,k_2)\bigr)f_{k_1k_2}(g)X_{k_1k_2}
\]
must be equal to
\[
\sigma_g(X_{k_1})\sigma_g(X_{k_2})=\bigl(f_{k_1}(g)X_{k_1}\bigr)\bigl(f_{k_2}(g)X_{k_2}\bigr)
=f_{k_1}(g)f_{k_2}(g)\tau(k_1,k_2)X_{k_1k_2},
\]
so \eqref{eq:fk1k2} holds. 

\medskip

\noindent(v) Since \eqref{eq:fk1k2} implies $[f_{k_1k_2}]=[f_{k_1}][f_{k_2}]$, we see that the mapping \eqref{eq:KH1} 
is a group homomorphism. Since $Z(\cD)\subset Z(\cC)$, the support of $Z(\cD)$ is contained in $K$. For $k\in K$, we 
have $[f_k]=1$ if and only if there is an element $l\in\LL^\times$ such that $f_k=\dc l$. 
By part (ii), for the element $X_k'=l^{-1}X_k$, we have $f_k'=(\dc l)^{-1}f_k$, so the condition $f_k=\dc l$ is equivalent 
to $f_k'(g)=1$ for all $g\in G$, and this latter, in view of \eqref{eq:sigma} and \eqref{eq:fks}, 
is equivalent $X_k'\in Z(\cD)$. 

\medskip

\noindent(vi) Since $G$ is abelian, $Z(\cC)$ and $Z(\cD)$ are graded subalgebras. By \cite[Lemma 4.2.2]{Allison_et_al}, 
a $G$-graded-simple algebra is simple if and only if its center is a field. 
Hence, for $\cC$ and $\cD$, centrality implies simplicity. It is clear that $\beta(k,K)=1$ if and only if $X_k\in Z(\cC)$,
so (b) and (c) are equivalent. By Corollary \ref{co:DCK}, (a) implies (b). Conversely, suppose $Z(\cC)=\LL$. 
Since the elements of $Z(\cD)$ must centralize $\cD_e$, we have $Z(\cD)\subset Z(\cC)$, so $Z(\cD)\subset\LL$.
But now \eqref{eq:barsigma} implies that $Z(\cD)\subset\LL^G$, so $Z(\cD)=\FF$. 
\end{proof}

\begin{corollary}\label{co:Kbeta}
If $\cD$ is central simple over $\FF$, then $\FF$ contains the primitive roots of unity of degree $\exp(K)$, 
the exponent of the finite abelian group $K$, and $K$ is isomorphic to $A\times A$ for some finite abelian group $A$.
\end{corollary}
\begin{proof}
Since $\beta:K\times K\rightarrow \FF^\times$ in this case is nondegenerate, it induces a group monomorphism from $K$ 
to the group $\Hom(K,\FF^\times)$ of characters of $K$ with values in $\FF^\times$. 
In particular, $\lvert \Hom(K,\FF^\times)\rvert\ge\lvert K\rvert$. 
Since $K$ is a finite direct product of finite cyclic groups, we always have 
$\lvert\Hom(K,\FF^\times)\rvert\le\lvert K\rvert$, and the equality holds if and only if $\FF^\times$ contains 
a primitive root of unity of degree $\exp(K)$.

Moreover, since $\beta$ is alternating, $K$ admits a ``symplectic basis'' (see e.g. \cite[2,\S 2]{EKmon}), 
i.e., a generating set of the form  $a_1,b_1,\ldots,a_m,b_m$ with the order of both $a_i$ and $b_i$ equal to some 
$n_i\ge 2$, $i=1,\ldots,m$, such that 
\begin{equation}\label{eq:symplectic_basis}
K=\langle a_1\rangle\times\langle b_1\rangle\times\cdots\times\langle a_m\rangle\times\langle b_m\rangle,
\end{equation}
and $\beta(a_i,b_i)=\zeta_i$, with $\zeta_i$ a primitive root of unity of degree $n_i$, 
while $\beta(a_i,b_j)=1$ for $i\neq j$ and $\beta(a_i,a_j)=\beta(b_i,b_j)=1$ for all $i,j$.
In particular, $K$ is the direct product of two isomorphic subgroups:
$\langle a_1,\ldots,a_m\rangle$ and $\langle b_1,\ldots,b_m\rangle$. 
\end{proof}

To get more precise information, we recall the \emph{inflation-restriction exact sequence} 
(coming from the Lyndon-Hochschild-Serre spectral sequence \cite[IX, (10.6)]{MacLane}):
\begin{equation}\label{eq:inf_res}
\begin{split}
1&\longrightarrow \Hc^1\bigl(G/K,(\LL^\times)^K\bigr)\stackrel{\inf}\longrightarrow \Hc^1(G,\LL^\times)
\stackrel{\res}\longrightarrow \Hc^1(K,\LL^\times)^{G/K}\\
&\stackrel{\rho}\longrightarrow \Hc^2\bigl(G/K,(\LL^\times)^K\bigr)\stackrel{\inf}\longrightarrow \Hc^2(G,\LL^\times)
\end{split}
\end{equation}
Since $K$ acts trivially on $\LL^\times$, $\Hc^1\bigl(G/K,(\LL^\times)^K\bigr)=\Hc^1(G/K,\LL^\times)$, 
which is trivial by Hilbert's Theorem $90$. On the other hand, since $G$ is abelian, we have
\[
\Hc^1(K,\LL^\times)^{G/K}=\Hc^1(K,\LL^\times)^G=\Hc^1\bigl(K,(\LL^\times)^G\bigr)=\Hc^1(K,\FF^\times)=\Hom(K,\FF^\times),
\]
and hence the restriction map:
\begin{equation}\label{eq:res}
\begin{split}
\res: \Hc^1(G,\LL^\times)&\longrightarrow \Hom(K,\FF^\times)\\
   [\gamma]\quad &\mapsto\qquad \gamma\vert_K
\end{split}
\end{equation}
is a well-defined group monomorphism. 

\begin{corollary}\label{co:betares}
The homomorphisms $f$ in \eqref{eq:KH1} and $\res$ in \eqref{eq:res} fit in the commutative diagram
\begin{equation}\label{eq:betares}
\begin{tikzcd}
\Hc^1(G,\LL^\times)\arrow[r, "\res"]&\Hom(K,\FF^\times)\\
&K\arrow[ul, "f"]\arrow[u, "\acute{\beta}"']
\end{tikzcd}
\end{equation}
where $\acute{\beta}$ is induced by the bicharacter $\beta$ as follows:
\begin{equation}\label{eq:acutebeta}
\acute{\beta}: k\mapsto \beta(\cdot,k).
\end{equation}
If $\cD$ is central simple over $\FF$, then all homomorphisms in this diagram are isomorphisms, and $\sigma$ is an 
isomorphism from $G$ onto the group of automorphisms of $\cC$ as a $K$-graded algebra over $\FF$. 
\end{corollary}
\begin{proof}
The commutativity of the diagram is clear from \eqref{eq:betaf}, and $\res$ is injective. 
If $\cD$ is central simple, then $f$ is also injective, so 
$\lvert K\rvert\leq \lvert\Hc^1(G,\LL^\times)\rvert\leq\lvert\Hom(K,\FF^\times)\rvert$, but 
$\lvert\Hom(K,\FF^\times)\rvert\leq \lvert K\rvert$, so we must have $\lvert K\rvert = 
\lvert\Hc^1(G,\LL^\times)\rvert =\lvert\Hom(K,\FF^\times)\rvert$.

To prove the assertion about $\sigma$, suppose $\sigma_g=\id_\cC$ for some $g\in G$. Then $\bar{\sigma}_g=\id_\LL$, so
$g\in K$. But for $g\in K$, we have $\sigma_g(X_k)=\beta(g,k)X_k$ for all $k\in K$. 
Hence $\beta(g,K)=1$, so $g=e$ by the nondegeneracy of $\beta$. Now, every automorphism of $\cC$ as a graded algebra
over $\LL$ is given by a character: $X_k\mapsto\chi(k)X_k$ for some $\chi\in\Hom(K,\LL^\times)$. 
But $\FF$ contains a primitive root of unity of degree $\exp(K)$, so $\Hom(K,\LL^\times)=\Hom(K,\FF^\times)$. 
Since $\grave{\beta}:K\to\Hom(K,\FF^\times)$, sending $k\mapsto\beta(k,\cdot)$, is an isomorphism, we conclude that 
$\sigma$ maps $K$ onto the automorphism group of $\cC$ as a $K$-graded algebra over $\LL$. It remains to recall that
$\bar{\sigma}$ maps $G$ onto $\Gal(\LL/\FF)$, and all $\bar{\sigma}_g$-semilinear automorphisms are compositions of one 
of them with linear automorphisms. 
\end{proof}

\begin{remark}
Commutative diagram \eqref{eq:betares} shows more: the kernel of $f$ is equal to the radical of $\beta$ (i.e., 
the kernel of $\acute{\beta}$), which implies that $Z(\cC)\simeq\LL\otimes_\FF Z(\cD)$ as graded algebras over $\LL$.
\end{remark}


\section{Simple abelian Galois extensions}\label{se:SimpleGal}

Let $G$ be a finite abelian group. It follows from Theorems \ref{th:PP} and \ref{th:GTCD} that any $G$-Galois extension 
of $\FF$ is isomorphic to an algebra of the form 
$\Ind_T^G(\cC)$ where $\cC=\Cent_\cD(\cD_e)$ and $\cD=\bigoplus_{g\in G}\cD_g$ is a central simple 
$G$-graded-division algebra with support $T$.  

This section is devoted to studying the simple $G$-Galois extensions of $\FF$, i.e., the $G$-algebras of the form
\[
\cC=\Cent_\cD(\cD_e),
\]
where $\cD$ is a central simple graded-division algebra with support $T=G$. 
We continue using notation from the previous section.
In particular, $\LL/\FF$ is a Galois field extension and $\bar\sigma:G\rightarrow \Gal(\LL/\FF)$ is a surjective homomorphism with kernel $K$. The ground field $\FF$ contains the primitive roots of unity of order $\exp(K)$.

\subsection{A structure theorem}
First we take a closer look at equation \eqref{eq:fk1k2}. Consider the short exact sequence 
\begin{equation}\label{eq:shortBZH}
1\longrightarrow \Bc^1(G,\LL^\times)\longrightarrow \Zc^1(G,\LL^\times)\longrightarrow \Hc^1(G,\LL^\times)
\longrightarrow 1\,,
\end{equation}
of abelian groups, which we will temporarily denote by $\Bc^1$, $\Zc^1$ and $\Hc^1$ for brevity.
The group $\Bc^1=\{\dc l: g\mapsto \bar\sigma_g(l)l^{-1}\mid l\in\LL^\times\}$ lies in the short 
exact sequence
\[
1\longrightarrow \FF^\times\longrightarrow\LL^\times\stackrel{\dc}\longrightarrow 
\Bc^1\longrightarrow 1\,,
\]
which is isomorphic to the short exact sequence 
\[
1\longrightarrow \FF^\times\longrightarrow \LL^\times\stackrel{\pi}\longrightarrow \LL^\times/\FF^\times
\longrightarrow 1\,,
\]
where $\pi$ is the natural homomorphism, by means of the maps $\id_{\FF^\times}$, $\id_{\LL^\times}$, and
$\eta:\Bc^1\to\LL^\times/\FF^\times$ sending $\dc l\mapsto l\FF^\times$. 

Hence, \eqref{eq:shortBZH} induces the following long exact sequence:
\[
1\rightarrow\Hom(K,\LL^\times/\FF^\times)\rightarrow \Hom(K,\Zc^1)\rightarrow
\Hom(K,\Hc^1)\stackrel{\delta}\rightarrow \Ext(K,\LL^\times/\FF^\times)\rightarrow \cdots
\]
For any abelian groups $A$ and $B$, we may identify $\Ext(A,B)$ with the symmetric 
cohomology group $\Hc^2_{\textup{sym}}(A,B)\bydef\Zc^2_{\textup{sym}}(A,B)/\Bc^2(A,B)$, where 
$\Zc^2_{\textup{sym}}(A,B)$ is the subgroup of symmetric $2$-cocycles of $A$ with values in $B$ 
(with trivial action of $A$ on $B$). Under this identification, the connecting homomorphism $\delta$ 
above becomes \cite[III, Lemma 1.2 and Theorem 9.1]{MacLane} the homomorphism
\begin{equation}\label{eq:Delta}
\delta:\Hom(K,\Hc^1(G,\LL^\times))\longrightarrow \Hc^2_{\textup{sym}}(K,\LL^\times/\FF^\times)\,
\left(\leq \Hc^2(K,\LL^\times/\FF^\times)\right)
\end{equation}
that takes any homomorphism $f:K\rightarrow \Hc^1(G,\LL^\times)$, $k\mapsto [f_k]$, 
to the class of the (symmetric) $2$-cocycle $\eta\circ\gamma:K\times K\to\LL^\times/\FF^\times$
where the $2$-cocycle $\gamma:K\times K\to\Bc^1(G,\LL^\times)$ is defined by the equation
$f_{k_1}f_{k_2}=\gamma(k_1,k_2)f_{k_1k_2}$.

Then \eqref{eq:fk1k2} tells us that $\gamma(k_1,k_2)=\dc\bigl(\tau(k_1,k_2)\bigr)$, so 
\begin{equation}\label{eq:etatau}
(\eta\circ\gamma)(k_1,k_2)=\tau(k_1,k_2)\FF^\times
\end{equation}
by definition of the isomorphism $\eta:\Bc^1(G,\LL^\times)\to\LL^\times/\FF^\times$. 
Therefore, the homomorphism $f:K\rightarrow \Hc^1(G,\LL^\times)$ 
and the class $[\tau]\in\Hc^2(K,\LL^\times)$ are related as follows:
$\delta(f)=\pi_*(\bigl([\tau]\bigr)$,
where $\pi_*:[\tau]\mapsto[\pi\circ\tau]$ is the homomorphism 
\begin{equation}\label{eq:pi*}
\pi_*:\Hc^2(K,\LL^\times)\longrightarrow \Hc^2(K,\LL^\times/\FF^\times)
\end{equation}
induced by the natural homomorphism $\pi:\LL^\times\to\LL^\times/\FF^\times$. 
We have proved one direction of the following result:

\begin{lemma}\label{lm:pi*Delta}
For a group homomorphism $f:K\rightarrow \Hc^1(G,\LL^\times)$ and a $2$-cocycle 
$\tau\in\Zc^2(K,\LL^\times)$, the equation 
\begin{equation}\label{eq:tauf}
\delta(f)=\pi_*\bigl([\tau]\bigr)
\end{equation}
holds if and only if there are 
$1$-cocycles $f_k\in\Zc^1(G,\LL^\times)$, for all $k\in K$, such that $f(k)=[f_k]$ and \eqref{eq:fk1k2} holds:
$f_{k_1}f_{k_2}=\dc\bigl(\tau(k_1,k_2)\bigr)f_{k_1k_2}$.
\end{lemma}
\begin{proof}
For the remaining direction, suppose that \eqref{eq:tauf} holds and pick, for each $k\in K$, some element 
$f_k\in\Zc^1(G,\LL^\times)$ such that $f(k)=[f_k]$ . 
As above, $\delta(f)=[\eta\circ\gamma]$ where $\gamma$ is defined by 
the equation $f_{k_1}f_{k_2}=\gamma(k_1,k_2)f_{k_1k_2}$. 
Hence, \eqref{eq:tauf} means that \eqref{eq:etatau} holds up to a coboundary in $\Bc^2(K,\LL^\times/\FF^\times)$, 
i.e., there exist elements $l_k\in\LL^\times$, for all $k\in K$, such that 
\[
(l_{k_1}l_{k_2}l_{k_1k_2}^{-1}\FF^\times)(\eta\circ\gamma)(k_1,k_2)=\tau(k_1,k_2)\FF^\times.
\]
Applying $\eta^{-1}$ to both sides and plugging in the definition of $\gamma$, we obtain:
\[
\dc(l_{k_1}l_{k_2}l_{k_1k_2}^{-1})f_{k_1}f_{k_2}f_{k_1k_2}^{-1}=\dc\bigl(\tau(k_1,k_2)\bigr).
\]
Hence, with $f_k'\bydef (\dc l_k)f_k$, we have $f(k)=[f_k']$ and 
$f_{k_1}'f_{k_2}'=\dc\bigl(\tau(k_1,k_2)\bigr)f_{k_1k_2}'$, as required.
\end{proof}

\begin{lemma}\label{lm:tau_exists}
Assume that the map $\res$ in \eqref{eq:res} is bijective. Then,
for any alternating bicharacter $\beta:K\times K\to\LL^\times$, 
there exists a $2$-cocycle $\tau\in\Zc^2(K,\LL^\times)$ such that \eqref{eq:tauf} holds 
for $f\bydef\res^{-1}\circ\acute{\beta}$ and $\beta(k_1,k_2)=\tau(k_1,k_2)\tau(k_2,k_1)^{-1}$ for all $k_1,k_2\in K$. 
Moreover, for any $1$-cocycles $f_k\in\Zc^1(G,\LL^\times)$ with $f(k)=[f_k]$ for all $k\in K$, the $2$-cocycle 
$\tau$ can be chosen to satisfy \eqref{eq:fk1k2}.
\end{lemma}

\begin{proof}
Denote by $\textup{alt}$ the map sending a $2$-cocycle $\tau$ to the alternating bicharacter 
$(k_1,k_2)\mapsto\tau(k_1,k_2)\tau(k_2,k_1)^{-1}$. As already mentioned, the result depends only on the class 
$[\tau]$, so we obtain a homomorphism $\Hc^2(K,\LL^\times)\to\Hom(K\wedge K,\LL^\times)$,
which we also denote by $\textup{alt}$. It is well known that the sequence
\[
1\longrightarrow \Hc^2_{\textup{sym}}(K,\LL^\times)\longrightarrow \Hc^2(K,\LL^\times)
\stackrel{\textup{alt}}{\longrightarrow}\Hom(K\wedge K,\LL^\times)\longrightarrow 1
\]
is exact. Here is a proof for completeness: if we write $K$ as a direct product of cyclic subgroups generated by 
$a_1,\ldots,a_m$ then, for any alternating bicharacter $\beta:K\times K\to\LL^\times$, we can define 
a bicharacter $\tau:K\times K\to\LL^\times$ by
\[
\tau(a_i,a_j)=\begin{cases}
\beta(a_i,a_j)&\text{if }i<j;\\
1&\text{if }i\ge j.
\end{cases}
\]
Since the action is trivial, any bicharacter is a $2$-cocycle, and clearly $\textup{alt}(\tau)=\beta$. 
Thus we obtain a homomorphism $\Hom(K\wedge K,\LL^\times)\to\Hc^2(K,\LL^\times)$
that is a right inverse of $\textup{alt}$. (In fact, we can put any abelian group $B$ with trivial action of 
$K$ in place of $\LL^\times$, and the splitting is even natural in $B$.) 

Now, given $\beta$, pick a $2$-cocycle $\tau_0\in\Zc^2(K,\LL^\times)$ such that $\textup{alt}(\tau_0)=\beta$, and 
consider a $2$-cocycle $\gamma\in\Zc^2(K,\LL^\times/\FF^\times)$ with $[\gamma]=\delta(f)\pi_*([\tau_0])^{-1}$ in 
$\Hc^2(K,\LL^\times/\FF^\times)$.
Since $\FF^\times$ contains a primitive root of unity of degree $\exp(K)$, $\beta$ takes values in $\FF^\times$, 
so $\gamma$ is a symmetric $2$-cocycle. But $\Hc^2_{\textup{sym}}$ can be interpreted as $\Ext$, and there is no 
higher $\Ext$ for abelian groups, so the following sequence is exact:
\begin{equation}\label{eq:exact_FL}
1\to\Hom(K,\LL^\times/\FF^\times)\to\Hc^2_{\textup{sym}}(K,\FF^\times)\to\Hc^2_{\textup{sym}}(K,\LL^\times)
\stackrel{\pi_*}{\rightarrow} \Hc^2_{\textup{sym}}(K,\LL^\times/\FF^\times)\to 1,
\end{equation}
where we have used the fact that $\Hom(K,\FF^\times)\to\Hom(K,\LL^\times)$ is an isomorphism.
Therefore, we can find a 2-cocycle $\alpha\in\Zc^2_{\textup{sym}}(K,\LL^\times)$ such that $\pi_*([\alpha])=[\gamma]$.
Then $\tau\bydef\tau_0\alpha$ satisfies \eqref{eq:tauf} and $\textup{alt}(\tau)=\beta$, as required.

Finally, given $1$-cocycles $f_k$ with $f(k)=[f_k]$, we can find elements $l_k\in\LL^\times$ as in the proof of 
Lemma \ref{lm:pi*Delta} and use them to modify $\tau$ rather than $f_k$: 
the $2$-cocycle $\tau'(k_1,k_2)\bydef\tau(k_1,k_2)l_{k_1}^{-1}l_{k_2}^{-1}l_{k_1k_2}$ 
will satisfy $f_{k_1}f_{k_2}=\dc\bigl(\tau'(k_1,k_2)\bigr)f_{k_1k_2}$, for all $k_1,k_2\in K$, 
and $\textup{alt}(\tau')=\beta$.
\end{proof}

\begin{remark}
The exact sequence \eqref{eq:exact_FL} also shows that, for a given $\beta$, not all $2$-cocycles 
$\tau\in\Zc^2(K,\LL^\times)$ with $\textup{alt}(\tau)=\beta$ will satisfy \eqref{eq:tauf}, 
unless $\Hc^2_{\textup{sym}}(K,\LL^\times/\FF^\times)$ happens to be trivial. 
The classes of $2$-cocycles satisfying \eqref{eq:tauf} form a coset of the image 
of $\Hc^2_{\textup{sym}}(K,\FF^\times)$ in $\Hc^2(K,\LL^\times)$.
\end{remark}

\begin{theorem}\label{th:main}
Let $G$ be a finite abelian group and let $\FF$ be a field. 

\begin{enumerate}

\item Any simple $G$-Galois extension of $\FF$ is isomorphic (as a $G$-algebra) to a twisted group algebra 
$\LL^\tau K$ for a subgroup $K$ of $G$, a Galois field extension $\LL/\FF$ with $\Gal(\LL/\FF)$ isomorphic to $G/K$, 
and a $2$-cocycle $\tau\in\Zc^2(K,\LL^\times)$ such that the following conditions hold:
\begin{enumerate}
\item[(i)] With respect to the $G$-action $\bar\sigma:G\to\Gal(\LL/\FF)$ given by 
the identification $G/K\simeq \Gal(\LL/\FF)$, we have $\lvert \Hc^1(G,\LL^\times)\rvert=\lvert K\rvert$ 
or, equivalently, $\FF$ contains a primitive root of unity of degree $\exp(K)$ 
and the monomorphism $\res$ in \eqref{eq:res} is surjective (hence bijective);

\item[(ii)] The alternating bicharacter $\beta(k_1,k_2)\bydef\tau(k_1,k_2)\tau(k_2,k_1)^{-1}$, whose values are 
automatically in $\FF^\times$ by the previous condition, is nondegenerate 
(so $\acute{\beta}$ in \eqref{eq:acutebeta} is bijective);

\item[(iii)] The isomorphism $f\bydef\res^{-1}\circ\acute{\beta}$ satisfies $\delta(f)=\pi_*\bigl([\tau]\bigr)$,
where $\delta$ and $\pi_*$ are as in \eqref{eq:Delta} and \eqref{eq:pi*}. 
\end{enumerate}
The $G$-action on $\LL^\tau K$ is given by $\sigma:G\to\Aut_\FF(\LL^\tau K)$ as follows:
\begin{equation}\label{eq:sigma_action}
\sigma_g(lX_k)=\bar\sigma_g(l)f_k(g)X_k\quad\forall g\in G,\,k\in K,\,l\in\LL,
\end{equation}
where $f_k\in\Zc^1(G,\LL^\times)$ is a representative of $f(k)\in\Hc^1(G,\LL^\times)$.

\item Conversely, given $K$, $\LL$ and $\tau$ satisfying the above conditions, the twisted group algebra $\LL^\tau K$
becomes a simple $G$-Galois extension of $\FF$ if we define the $G$-action by \eqref{eq:sigma_action}
where the representatives $f_k$ are chosen (by Lemma \ref{lm:pi*Delta}) so that 
$f_{k_1}f_{k_2}=\dc\bigl(\tau(k_1,k_2)\bigr)f_{k_1k_2}$ for all $k_1,k_2\in K$. 
\end{enumerate}

\end{theorem}

\begin{proof}
(1) This has already been proved: see Proposition \ref{pr:Kbeta}, its corollaries, and Lemma \ref{lm:pi*Delta}.

(2) Let $\cA=\LL^\tau K$, with $K$, $\LL$, and $\tau$ satisfying conditions (i), (ii), and (iii).
The nondegeneracy of $\beta$ shows that $Z(\cA)=\LL$. Now, $\cA$ is naturally $K$-graded, 
and as such it is a graded-division algebra, so we can apply \cite[Lemma 4.2.2]{Allison_et_al}, as in the proof of
part (vi) of Proposition \ref{pr:Kbeta}, to conclude that $\cA$ is simple.

For any $k\in K$, we have $f_k\in \Zc^1(G,\LL^\times)$ and hence the computation in the proof of part (i) of 
Proposition \ref{pr:Kbeta} shows that $\sigma_{g_1g_2}=\sigma_{g_1}\sigma_{g_2}$. 
Also, since $f_{k_1}f_{k_2}=\dc\bigl(\tau(k_1,k_2)\bigr)f_{k_1k_2}$ for all $k_1,k_2\in K$, the computation 
in the proof of part (iv) of that result shows that $\sigma_g$ is an automorphism of $\cA$ for any $g\in G$.
Thus, we have a well-defined $G$-action on $\cA$. Note that if $g\in K$ then the automorphism $\sigma_g$ is $\LL$-linear
and $\sigma_g(X_k)=\beta(g,k)X_k$ because $f_k\vert_K=\acute{\beta}(k)=\beta(\cdot,k)$ by definition of $f$.

We have to prove that $\cA^G=\FF$ and that the algebra map $\Phi:\cA\#  \FF G\rightarrow \End_\FF(\cA)$
sending $ag\mapsto\bigl(b\mapsto a\sigma_g(b)\bigr)$ is an isomorphism.

The action of $G$ on $\cA$ leaves the components of the $K$-grading invariant, hence in 
order to compute $\cA^G$, it is sufficient to consider homogeneous elements. 
Let $l\in\LL^\times$ and $k\in K$ such that $lX_k\in\cA^G$. 
The computation in the proof of part (ii) of Proposition \ref{pr:Kbeta} shows that, for all $g\in G$, 
$\sigma_g(lX_k)=(\dc l)(g)f_k(g)lX_k$, hence $f_k=(\dc l)^{-1}$ and, since $f$ is injective, $k=e$ and $f_e=(\dc l)^{-1}$.
But all automorphisms $\sigma_g$ must send the identity element $1_\cA=\tau(e,e)^{-1}X_e$ to itself, 
so $f_e=\dc\bigl(\tau(e,e)\bigr)$. It follows that $l\FF^\times=\tau(e,e)^{-1}\FF^\times$, 
and hence $lX_e\in\FF^\times 1_\cA$. 

Finally, consider the group $\cA_{\textup{gr}}^\times$ of nonzero homogeneous elements of $\cA$ 
(for the natural $K$-grading). For any $g\in G$, consider the map:
\[
\begin{split}
\chi_g: \cA_{\textup{gr}}^\times&\longrightarrow \LL^\times\\
   a\ &\mapsto\ \sigma_g(a)a^{-1}.
\end{split}
\]
For $k\in K$ and $0\neq a\in \cA_k$, we have $\sigma_g(a)\in\cA_k=\LL a$, so $\chi_g$ is well defined. 
(Explicitly, $\chi_g(lX_k)=(\dc l)(g)f_k(g)$ for any $g\in G$.)
Moreover, $\chi_g(ab)=\sigma_g(ab)(ab)^{-1}=\sigma_g(a)\bigl(\sigma_g(b)b^{-1}\bigr)a^{-1}=\chi_g(a)\chi_g(b)$, so
$\chi_g$ is a character. For any elements $g_1\ne g_2$ in $G$, let us check that $\chi_{g_1}\neq \chi_{g_2}$.
Clearly, this is equivalent to $\sigma_{g_1}\ne\sigma_{g_2}$, so we have to show that the $G$-action is faithful: 
$\sigma_g\ne\id_\cA$ for all $g\ne e$. If $g\notin K$, then already $\sigma_g\vert_\LL\neq \id_\LL$. 
But if $e\neq g\in K$, then $\sigma_g(X_k)=\beta(g,k)X_k$, so $\sigma_g\ne\id_\cA$ because $\beta(g,K)\ne 1$
by the nondegeneracy of $\beta$. 

Now, the $K$-grading on $\cA$ induces a $K$-grading on $\cA\#\FF G$ and on $\End_\FF(\cA)$, and 
$\Phi$ is a homomorphism of $K$-graded algebras. Since $\dim_\FF(\cA\#\FF G)=\lvert G\rvert^2=\dim_\FF\End_\FF(\cA)$, 
it suffices to prove that $\Phi$ is injective. To this end, suppose $X\in\ker\Phi\cap\cA_k$, for some $k\in K$. 
Then $X=\sum_{g\in G}(l_gX_k)g$, with $l_g\in \LL$ for all $g\in G$, and we get
$0=\Phi(X)=\Phi(X_k)\Phi\bigl(\sum_{g\in G}l_gg\bigr)$. Since $X_k$ is invertible, we get 
$0=\Phi\bigl(\sum_{g\in G}l_gg\bigr)=\sum_{g\in G}l_g\sigma_g$. Hence, for any $a\in\cA_{\textup{gr}}^\times$, 
we have $0=\bigl(\sum_{g\in G}l_g\sigma_g(a)\bigr)a^{-1}=\sum_{g\in G}l_g\chi_g(a)$. 
By the linear independence of characters, we conclude that $l_g=0$ for all $g\in G$, so $X=0$.
\end{proof}

\subsection{Classification up to isomorphism}
We can now obtain a classification of simple $G$-Galois extensions of $\FF$, but first we need to introduce some notation.
Fix an algebraic closure $\overline{\FF}$ of $\FF$. Then every finite field extension of $\FF$ is isomorphic, 
over $\FF$, to a subfield of $\overline{\FF}$. It follows that, for every pair $(\LL,\theta)$, where $\LL$ is a finite 
\emph{abelian} Galois field extension of $\FF$ and $\theta:G\rightarrow \Gal(\LL/\FF)$ 
is a group homomorphism, there is a \emph{unique} isomorphic pair $(\LL',\theta')$ with $\LL'\subset\overline{\FF}$. 
Here isomorphism is understood in the sense of $G$-algebras over $\FF$: there exists an algebra isomorphism 
$\psi:\LL\rightarrow \LL'$ such that $\theta'_g=\psi\circ\theta_g\circ\psi^{-1}$ for all $g\in G$.

Consider the set $\mathfrak{Z}_\FF(G)$ of pairs $(\LL,\theta)$ where
\begin{enumerate}
\item[(1)] $\LL\subset\overline{\FF}$ is a finite Galois extension of $\FF$;

\item[(2)] $\theta:G\rightarrow \Gal(\LL/\FF)$ is a surjective group homomorphism such that, for
$K\bydef\ker\theta$, the following conditions hold:
\begin{enumerate}
\item[(i)] $K$ admits a nondegenerate alternating bicharacter with values in $\FF^\times$ or, equivalently, 
$K$ is isomorphic to $A\times A$ for some abelian group $A$ and 
$\FF$ contains a primitive root of unity of degree $\exp(K)$;
\item[(ii)] Every character $K\to\FF^\times$ can be extended to a $1$-cocycle $G\to\LL^\times$ or, equivalently, the 
``transgression'' map $\rho:\Hom(K,\FF^\times)\to\Hc^2(G/K,\LL^\times)$ in the exact sequence \eqref{eq:inf_res} is 
trivial. 
\end{enumerate}
\end{enumerate}

By Lemma \ref{lm:tau_exists} and Theorem \ref{th:main}, $\mathfrak{Z}_\FF(G)$ is a set of representatives for the 
isomorphism classes of the centers of simple $G$-Galois extensions of $\FF$. We note that (ii) is satisfied if 
every character $K\to\FF^\times$ extends to a character $G\to\FF^\times$ (for example, if $\FF$ contains a primitive root 
of unity of degree $\exp(G)$ or if $K$ is a direct summand of $G$), but this condition is not necessary 
(see Example \ref{ex:classes_of_Gal_ext}(c)).

For every pair $(\LL,\theta)\in\mathfrak{Z}_\FF(G)$, fix an extension of every character $K\to\FF^\times$
to a $1$-cocycle $G\to\LL^\times$ and denote by $\Xi(\LL,\theta)$ the resulting subset of $\Zc^1(G,\LL^\times)$; 
it is a transversal for the subgroup $\Bc^1(G,\LL^\times)\simeq\LL^\times/\FF^\times$.

Let $\mathfrak{T}_\FF(G)$ be the set of triples $(\LL,\theta,\xi)$ where $(\LL,\theta)\in\mathfrak{Z}_\FF(G)$ and 
\begin{enumerate}
\item[(3)] $\xi=\tau\Bc^2(K,\FF^\times)\in\Zc^2(K,\LL^\times)/\Bc^2(K,\FF^\times)$ such that the alternating bicharacter 
$\beta(k_1,k_2)\bydef\tau(k_1,k_2)\tau(k_2,k_1)^{-1}$, which depends only on the class $[\tau]=\tau\Bc^2(K,\LL^\times)$
in $\Hc^2(K,\LL^\times)$, is nondegenerate, and the following equation holds: 
$f_{k_1}f_{k_2}=\dc\bigl(\tau(k_1,k_2)\bigr)f_{k_1k_2}$ for all $k_1,k_2\in K$, 
where $f_k$ is the unique element of $\Xi(\LL,\theta)$ such that $f_k\vert_K=\beta(\cdot,k)$. 
\end{enumerate}

Let $\cC$ be a simple $G$-Galois extension of $\FF$. Recall that $\cC$ has a natural $G$-grading defined by \eqref{eq:MUgr}
(the ``Miyashita-Ulbrich grading'', see Remark \ref{re:Miyashita-Ulbrich}). The support $K$ of this grading is the kernel 
of the $G$-action on the center of $\cC$, and $\cC$ is a graded-division algebra with $\cC_e=Z(\cC)$.
In particular, the nondegenerate alternating bicharacter $\beta$ on $K$ with values in $\FF^\times$ 
is an invariant of $\cC$: $\beta(k_1,k_2)=c_1c_2c_1^{-1}c_2^{-1}$ for any nonzero $c_1\in\cC_{k_1}$, $c_2\in\cC_{k_2}$,
$k_1,k_2\in K$. 

We define $\Psi(\cC)\in\mathfrak{T}_\FF(G)$ to be the following triple $(\LL,\theta,\xi)$:
\begin{enumerate}
\item $\LL$ is the unique subfield of $\overline{\FF}$ that is isomorphic to $Z(\cC)$ over $\FF$; 
\item $\theta$ is the unique homomorphism $G\to\Gal(\LL/\FF)$ such that $(\LL,\theta)$ is isomorphic to 
$(Z(\cC),\bar{\sigma})$ where $\bar{\sigma}_g\bydef\sigma_g\vert_{Z(\cC)}$ and $\sigma_g\in\Aut_\FF(\cC)$ is the action 
of $g\in G$ on $\cC$;
\item $\xi\bydef\tau\Bc^2(K,\FF^\times)$ where $K=\ker\bar{\sigma}$, 
$\tau(k_1,k_2)\bydef\iota^{-1}(X_{k_1}X_{k_2}X_{k_1k_2}^{-1})$, $\iota:\LL\to Z(\cC)$ is an isomorphism over $\FF$, 
and $X_k\in\cC_k$ are nonzero elements chosen in such a way that $\sigma_g(X_k)=\iota(f_k(g))X_k$, 
where $f_k$ is the unique element of $\Xi(\LL,\theta)$ such that $f_k\vert_K=\beta(\cdot,k)$. 
\end{enumerate}
$\Psi(\cC)$ is well defined. Indeed, the triple $(\LL,\theta,\xi)$ satisfies all required conditions 
by Proposition \ref{pr:Kbeta} and its corollaries. In particular, part (ii) of Proposition \ref{pr:Kbeta} 
shows that the elements $X_k$ as above exist and are unique up to factors in $\FF^\times$, 
so the $2$-cocycle $(k_1,k_2)\mapsto X_{k_1}X_{k_2}X_{k_1k_2}^{-1}$ with values in $Z(\cC)$ is determined up to 
a coboundary in $\Bc^2(K,\FF^\times)$, and part (iv) shows that $f_{k_1}f_{k_2}=\dc\bigl(\tau(k_1,k_2)\bigr)f_{k_1k_2}$ 
for all $k_1,k_2\in K$. Finally, different choices of the isomorphism $\iota:\LL\to Z(\cC)$ 
produce the same coset $\xi$. Indeed, any two such isomorphisms $\iota$ and $\iota'$
differ by an element of $\Gal(\LL/\FF)$, so we can write $\iota'=\iota\circ\theta_h$ for some $h\in G$ thanks to 
the surjectivity of $\theta$. It follows that if the elements $X_k$ ($k\in K$) are chosen using $\iota$ 
then the elements $\sigma_h(X_k)$ are an allowable choice for $\iota'$:
\[
\begin{split}
\sigma_g(\sigma_h(X_k))&=\sigma_h(\sigma_g(X_k))=\sigma_h\bigl(\iota(f_k(g))X_k\bigr)\\
&=\bar{\sigma}_h\bigl(\iota(f_k(g))\bigr)\sigma_h(X_k)=\iota\bigl(\theta_h(f_k(g))\bigr)\sigma_h(X_k)\\
&=\iota'(f_k(g))\sigma_h(X_k),
\end{split}
\]
for all $g\in G$ and $k\in K$, and these elements give 
\[
\begin{split}
\tau'(k_1,k_2)&=(\iota')^{-1}\bigl(\sigma_h(X_{k_1})\sigma_h(X_{k_2})\sigma_h(X_{k_1k_2})^{-1}\bigr)\\
&=\theta_h^{-1}\bigl(\iota^{-1}(\bar{\sigma}_h(X_{k_1}X_{k_2}X_{k_1k_2}^{-1}))\bigr)\\
&=\theta_h^{-1}\bigl(\theta_h(\iota^{-1}(X_{k_1}X_{k_2}X_{k_1k_2}^{-1}))\bigr)=\tau(k_1,k_2).
\end{split}
\]

\begin{corollary}\label{co:main}
Let $G$ be a finite abelian group and $\FF$ a field. Denote by $\uE_\FF^{\textup{simple}}(G)$ the set of 
isomorphism classes of simple $G$-Galois extensions of $\FF$.
Then the mapping $\uE_\FF^{\textup{simple}}(G)\rightarrow \mathfrak{T}_\FF(G)$ sending
$[\cC]_{\textup{$G$-alg}}\mapsto\Psi(\cC)$ is a bijection.
\end{corollary}

\begin{proof}
Let $\cC$ and $\cC'$ be simple $G$-Galois extensions of $\FF$. Denote $\Psi(\cC)=(\LL,\theta,\xi)$, 
$\Psi(\cC')=(\LL',\theta',\xi')$, and similarly for other parameters. 

If $\cC\simeq\cC'$, then $K=K'$, $\beta=\beta'$ and $(\LL,\theta)\simeq(\LL',\theta')$, 
hence, by construction, $(\LL,\theta)=(\LL',\theta')$ and $f_k=f'_k$ for all $k\in K$. 
Consider an isomorphism $\psi:\cC\to\cC'$. As we have seen, the isomorphisms $\iota:\LL\to Z(\cC)$ and 
$\iota':\LL\to Z(\cC')$ may be chosen arbitrarily, so we pick some $\iota$ and set 
$\iota'\bydef\psi\circ\iota$. 
It follows that if the elements $X_k$ ($k\in K$) are chosen using $\iota$ 
then the elements $\psi(X_k)$ are an allowable choice for $\iota'$:
\[
\begin{split}
\sigma'_g(\psi(X_k))&=\psi(\sigma_g(X_k))=\psi\bigl(\iota(f_k(g))X_k\bigr)\\
&=\psi\bigl(\iota(f_k(g))\bigr)\psi(X_k)=\iota'(f_k(g))\psi(X_k),
\end{split}
\]
for all $g\in G$ and $k\in K$, and these elements give 
\[
\begin{split}
\tau'(k_1,k_2)&=(\iota')^{-1}\bigl(\psi(X_{k_1})\psi(X_{k_2})\psi(X_{k_1k_2})^{-1}\bigr)\\
&=\iota^{-1}\bigl(\psi^{-1}(\psi(X_{k_1}X_{k_2}X_{k_1k_2}^{-1}))\bigr)=\tau(k_1,k_2).
\end{split}
\]
We have shown that $\Psi(\cC)=\Psi(\cC')$, so our mapping is well defined.

Conversely, suppose that $\Psi(\cC)=\Psi(\cC')$. Replacing $\cC$ and $\cC'$ by isomorphic copies, we may assume 
that $\LL=Z(\cC)=Z(\cC')$ and $\theta=\bar{\sigma}=\bar{\sigma}'$. Choose elements $X_k$ and $X'_k$ ($k\in K$) using 
$\iota=\iota'=\id_\LL$. Since $\tau$ and $\tau'$ differ by an element of $\Bc^2(K,\FF^\times)$, there exist elements 
$\lambda_k\in\FF^\times$ ($k\in K$) such that 
$\tau'(k_1,k_2)=\lambda_{k_1}\lambda_{k_2}\lambda_{k_1k_2}^{-1}\tau(k_1,k_2)$.
It is easy to verify that the mapping $\cC\to\cC'$ sending $\sum_{k\in K}l_k X_k\mapsto\sum_{k\in K}\lambda^{-1}_k l_k X'_k$
is an isomorphism of $G$-algebras.

The surjectivity of $\Psi$ follows from part (2) of Theorem \ref{th:main}.
\end{proof}

It is convenient to define a specific $G$-algebra in the isomorphism class $\Psi^{-1}(\LL,\theta,\xi)$:

\begin{df}\label{df:modelsC}
Given a triple $(\LL,\theta,\xi)\in \mathfrak{T}_\FF(G)$, let $\tau$ be a representative of the coset $\xi$.
Denote by $\cC(\LL,\theta,\tau)$ the following simple $G$-Galois extension of $\FF$: as an algebra, it is 
$\LL^\tau K$, with $K\bydef\ker\theta$, and the action of $G$ given by 
$g\cdot(lX_k)=\theta_g(l)f_k(g)X_k$, where $f_k$ is the unique element of $\Xi(\LL,\theta)$ 
such that $f_k\vert_K=\beta(\cdot,k)$. By abuse of notation, we will sometimes write $\cC(\LL,\theta,\xi)$ 
instead of $\cC(\LL,\theta,\tau)$.
\end{df}

For a pair $(\LL,\theta)\in\mathfrak{Z}_\FF(G)$ and a nondegenerate alternating bicharacter 
$\beta:{K\times K}\to\FF^\times$, there exists a $2$-cocycle $\tau_0\in\Zc^2(K,\LL^\times)$
that satisfies $\beta(k_1,k_2)=\tau_0(k_1,k_2)\tau_0(k_2,k_1)^{-1}$ and 
$f_{k_1}f_{k_2} = \dc\bigl(\tau_0(k_1,k_2)\bigr)f_{k_1k_2}$ for all $k_1,k_2\in K$ (Lemma \ref{lm:tau_exists}). 
Then all other such  $2$-cocycles have the form $\tau=\tau_0\alpha$ where $\alpha\in\Zc^2_{\textup{sym}}(K,\FF^\times)$, 
since the kernel of $\dc:\LL^\times\to\Bc^1(G,\LL^\times)$ is $\FF^\times$. 
Therefore, the isomorphism classes of simple $G$-Galois extensions of $\FF$ whose center is isomorphic to $(\LL,\theta)$ 
and commutation relations are given by $\beta$ are in bijection with 
$\Hc^2_{\textup{sym}}(K,\FF^\times)\simeq\Ext(K,\FF^\times)$. 

\begin{remark}\label{re:Alberto}
We can assign to each simple $G$-Galois extension $\cC$ of $\FF$ an easier invariant $\overline{\Psi}(\cC)$,
which is obtained from $\Psi(\cC)$ by replacing the third component of the triple by its image in $\Hc^2(K,\LL^\times)$.
However, this is not a complete invariant, in general. It classifies simple $G$-Galois extensions up to the following
equivalence relation: $\cC\sim\cC'$ if there exists an isomorphism of $G$-graded algebras $\psi:\cC\to\cC'$ that restricts
to an isomorphism of $G$-algebras $Z(\cC)\to Z(\cC')$. For an element $(\LL,\theta,\xi)\in\mathfrak{T}_\FF(G)$,
the set of isomorphism classes of simple Galois extensions of $\FF$ that are equivalent to $\cC(\LL,\theta,\xi)$ 
in this sense is in bijection with the triples $(\LL,\theta,\xi')$ where 
$\xi'$ belongs to the coset of $\xi$ with respect to 
the subgroup $\bigl(\Zc^2(K,\FF^\times)\cap\Bc^2(K,\LL^\times)\bigr)/\Bc^2(K,\FF^\times)$ of 
$\Zc^2(K,\LL^\times)/\Bc^2(K,\FF^\times)$. This subgroup is the kernel of the homomorphism 
$\Hc^2_{\textup{sym}}(K,\FF^\times)\to\Hc^2_{\textup{sym}}(K,\LL^\times)$ in the exact sequence \eqref{eq:exact_FL}, 
so it is equal to the image of the connecting homomorphism 
$\Hom(K,\LL^\times/\FF^\times)\to\Hc^2_{\textup{sym}}(K,\FF^\times)$ and therefore isomorphic to 
$\Hom(K,\LL^\times/\FF^\times)$. Actually, this latter group acts simply transitively on the above set of isomorphism
classes as follows: given a simple $G$-Galois extension $\cC$ of $\FF$ and a homomorphism 
$\lambda:K\to\Bc^1(G,\LL^\times)\simeq\LL^\times/\FF^\times$, we define $\cC_\lambda$ to be the algebra $\cC$ with the  
modified $G$-action $\sigma^\lambda:G\to\Aut_\FF(\cC)$ given by $\sigma^\lambda_g(c)\bydef\lambda_k(g)\sigma_g(c)$ for all 
$c\in\cC_k$, $k\in K$. In other words, $\sigma^\lambda_g=\sigma_{\kappa_\lambda(g)}\sigma_g$ where 
$\kappa_\lambda:G\to K$ is defined by $\beta(\kappa_\lambda(g),k)=\lambda_k(g)$ for all $k\in K$. 
Since homomorphisms $G\to\FF^\times$ are precisely the $1$-cocycles with values in $\FF^\times$, 
and a $1$-cocycle is a coboundary if and only if it has trivial restriction to $K$, 
we conclude that $\kappa_\lambda$ is a homomorphism with trivial restriction to $K$, 
and the mapping $\lambda\mapsto\kappa_\lambda$ yields an isomorphism $\Hom(K,\LL^\times/\FF^\times)\to\Hom(G/K,K)$. 
\end{remark}

\begin{examples}\label{ex:classes_of_Gal_ext} 
The following are special cases of simple $G$-Galois extensions of $\FF$:
\begin{enumerate}
\item[(a)] Galois field extensions $\LL$ of $\FF$ with $\Gal(\LL/\FF)\simeq G$: these correspond to the case $K=1$.

\item[(b)] Central simple graded-division algebras over $\FF$ with support $G$ and $1$-dimensional homogeneous components:
these correspond to the case $K=G$ and are parametrized by the elements of $\Hc^2(G,\FF^\times)$ such that the 
corresponding alternating bicharacter $G\times G\to\FF^\times$ is nondegenerate. If $\FF$ is algebraically closed, these 
are the only simple $G$-Galois extensions, and they are parametrized by nondegenerate alternating bicharacters.

\item[(c)] Suppose $\Br(\FF)$ is trivial (for example, $\FF$ is finite). 
Then, for any subgroup $K$ admitting a nondegenerate alternating bicharacter 
and any Galois field extension $\LL$ with $\Gal(\LL/\FF)\simeq G/K$, every central simple 
graded-division algebra over $\LL$ with support $K$ and $1$-dimensional homogeneous components admits a $G$-action 
that makes it a $G$-Galois extension of $\FF$ (with the given underlying grading). 
Up to isomorphism, these actions are parametrized by the set $\Aut(G/K)\times\Hom(K,\LL^\times/\FF^\times)$ 
(see Remark \ref{re:Alberto}). All simple $G$-Galois extensions have this form.

\item[(d)] Suppose $\FF$ is real closed (for example, $\FF=\RR$). Then $K$ must be $2$-elementary of even rank, and 
it follows from the classification of central simple graded-division algebras \cite{BZreal_simple,ARE} 
(and can also be shown by considering $1$-cocycles) that $G$ must be $2$-elementary in order for simple 
$G$-Galois extensions to exist. This condition is also sufficient.
\end{enumerate}
\end{examples}

\smallskip

Among the $G$-algebras, $G$-Galois extensions of $\FF$ can be characterized as follows:

\begin{corollary}\label{co:main2}
Let $G$ be a finite abelian group and $\FF$ a field. Let $\cA$ be an algebra over $\FF$
endowed with a $G$-action $\sigma:G\rightarrow \Aut_\FF(\cA)$. Then the $G$-algebra $\cA$ is a 
$G$-Galois extension of $\FF$ if and only if the following  conditions hold:
\begin{enumerate}
\item $\dim_\FF\cA=|G|$;
\item $\LL\bydef Z(\cA)$ is a $G/K$-Galois extension of $\FF$ where $K$ is the kernel of the homomorphism 
$\bar\sigma:G\rightarrow\Aut_\FF(\LL)$, $g\mapsto \sigma_g\vert_\LL$;

\item $\FF$ contains a primitive root of unity of degree $\exp(K)$;

\item For every $\chi\in\wh{K}\bydef\Hom(K,\FF^\times)$, the eigenspace 
\[
\cA_\chi\bydef\{a\in \cA\mid \sigma_k(a)=\chi(k)a\ \forall k\in K\}
\]
contains an invertible element.
\end{enumerate}
\end{corollary}

\begin{proof}
First we consider the special case where $\LL$ is a field.
If $\cA$ is a $G$-Galois extension of $\FF$, then Theorem \ref{th:main} shows that conditions (1) through (4) 
are satisfied, since in this case $\cA_\chi=\LL X_{\acute{\beta}^{-1}(\chi)}$.

Conversely, assume that these conditions hold. By condition (2), $\LL$ is a Galois field extension of $\FF$ 
with $\Gal(\LL/\FF)\simeq G/K$. Condition (3) shows that the commuting $\LL$-linear operators 
$\sigma_k$ ($k\in K$) can be simultaneously diagonalized, hence $\cA=\bigoplus_{\chi\in\wh{K}}\cA_\chi$ is a 
$\wh{K}$-grading on $\cA$ as an $\LL$-algebra. By condition (4), we can pick an invertible element $u_\chi\in\cA_\chi$ 
for each $\chi\in\wh{K}$. Since $|\wh{K}|=|K|$ by condition (3), the subspace $\bigoplus_{\chi\in\wh{K}}\LL u_\chi$ 
has $\FF$-dimension $|K|\dim_\FF\LL=|G|$, which is equal to $\dim_\FF\cA$ by condition (1). 
We conclude that $\cA_\chi=\LL u_\chi$ for all $\chi\in\wh{K}$, so $\cA$ is a $\wh{K}$-graded-division algebra, 
and $\cA^G=(\LL 1_\cA)^G=\FF 1_\cA$. 
It is clear that the action of $G$ on $\cA$ is faithful, so the argument at the end of the proof of part (2) of 
Theorem \ref{th:main} (with $\wh{K}$ playing the role of $K$) shows that $\cA$ is a $G$-Galois extension of $\FF$.

The general case reduces to the special case that we have considered using the functor $\Ind_T^G$. 
If $\cA\simeq\Ind_T^G(\cC)$ for a subgroup $T$ of $G$ and a $T$-algebra $\cC$ then, as a $T$-algebra, 
$\cA$ is isomorphic to the direct product of $[G:T]$ copies of $\cC$
and $Z(\cA)\simeq\Ind_T^G(Z(\cC))$ is isomorphic to the direct product of $[G:T]$ copies of $Z(\cC)$. 
It follows that conditions (1) through (4) hold for $\cA$ if and only if they hold for $\cC$ 
(with $T$ playing the role of $G$).

If $\cA$ is a $G$-Galois extension of $\FF$ then we know that $\cA\simeq\Ind_T^G(\cC)$ for a simple 
$T$-Galois extension $\cC$, and conditions (1) through (4) hold for $\cC$ by the special case considered above.
Conversely, let $\cA$ be a $G$-algebra satisfying these conditions. By condition (2), 
$\LL\simeq\Ind_{T/K}^{G/K}(\KK)$ for some Galois field extension $\KK$ of $\FF$, with 
$\Gal(\KK/\FF)\simeq T/K$, where $T/K$ is the stabilizer (under the $G/K$-action) 
of a primitive idempotent $\varepsilon$ of $\LL$. The primitive 
idempotents of $\LL$ give a decomposition of $\cA$ into a direct sum of ideals, which are permuted transitively by $G$.
A standard argument then shows that $\cA\simeq\Ind_T^G(\cC)$ for $\cC\bydef\varepsilon\cA$. 
Since $\cC$ satisfies conditions (1) through (4), it is a $T$-Galois extension of $\FF$ by the special case above. 
\end{proof}

\begin{remark}
For a simple $G$-Galois extension $\cA$ of $\FF$ as in Corollary \ref{co:main2}, the canonical $K$-grading is given as 
follows. Pick $0\ne u_\chi\in\cA_\chi$ and define
\[
\hat{\beta}(\chi,\psi)\bydef u_\chi u_\psi u_\chi^{-1} u_\psi^{-1}\quad\forall\chi,\psi\in\wh{K}.
\]
This is an alternating bicharacter, independent of the choice of the elements $u_\chi$. 
It takes values in $\FF$ by condition (3) and is nondegenerate since $Z(\cA)=\LL$. Hence $\hat{\beta}$
gives an isomorphism $K\to\wh{K}$: for any $k\in K$, we define $\varphi(k)$ to be the unique character $\chi$
such that $\hat{\beta}(\chi,\lambda)=\lambda(k)$ for all $\lambda\in\wh{K}$. We use $\varphi$ to convert the 
$\wh{K}$-grading on $\cA$ to a $K$-grading: $\cA=\bigoplus_{k\in K}\cC_k$ where $\cC_k\bydef\cA_{\varphi(k)}$.
We can also use $\varphi$ to transport $\hat{\beta}$ to $K$, i.e., we define $\beta:K\times K\to\FF^\times$ by 
\[
\beta(h,k)\bydef\hat{\beta}(\varphi(h),\varphi(k))=\varphi(k)(h)\quad\forall h,k\in K.
\]
It follows that, for any $b\in\cC_k$, we have $\sigma_h(b)=\varphi(k)(h)b=\beta(h,k)b=u_{\varphi(h)}bu_{\varphi(h)}^{-1}$,
so $\sigma_h(b)a=ab$ for all $a\in\cC_h$ and $b\in\cA$. Therefore, the grading $\cA=\bigoplus_{k\in K}\cC_k$ 
satisfies \eqref{eq:MUgr}. (Also, $\varphi=\acute{\beta}$.)
\end{remark}

\subsection{Construction with generators and relations}
For a fixed $(\LL,\theta)\in\mathfrak{Z}_\FF(G)$, the simple Galois extensions $\cC(\LL,\theta,\tau)$ 
as in Definition \ref{df:modelsC} can be explicitly described by means of generators and relations. 
Write $K\bydef\ker\theta$ as a direct product of cyclic subgroups generated by $a_1,\ldots, a_m$.
Then the twisted group algebra $\LL^\tau K$ is generated over $\LL$ by the elements $X_i\bydef X_{a_i}$, 
which satisfy the following relations: 
\[
X_i X_j=\beta_{ij}X_j X_i\text{ and }X_i^{o(a_i)}=\mu_i 1,
\]
where $\beta_{ij}\bydef\beta(a_i,a_j)\in\FF^\times$, $\mu_i\in\LL^\times$, 
and $o(g)$ denotes the order of a group element $g$. It is clear that these relations are defining.
Therefore, we can forget about the $2$-cocycle $\tau$ and express everything in terms of the bicharacter $\beta$ 
and the scalars $\mu_i$, $i=1,\ldots,m$. 
By definition of $\cC(\LL,\theta,\tau)$, the $G$-action on the generators is given by
\begin{equation}\label{eq:action_on_gen}
\sigma_g(X_i) = f_{a_i}(g) X_i\quad\forall g\in G,
\end{equation}
where $f_{a_i}$ is the fixed extension of the character $\beta(\cdot,a_i):K\to\FF^\times$ to a $1$-cocycle $G\to\LL^\times$.
Since $\sigma_g(X_i^{o(a_i)})=\sigma_g(X_i)^{o(a_i)}$, we have 
\begin{equation}\label{eq:mu_compat}
\theta_g(\mu_i)=f_{a_i}(g)^{o(a_i)}\mu_i\quad\forall g\in G.
\end{equation}
In other words, $\dc\mu_i=f_{a_i}^{o(a_i)}$, $i=1,\ldots,m$. 

Conversely, suppose $\beta:K\times K\to\FF^\times$ is an alternating nondegenerate bicharacter and $\mu_1,\ldots,\mu_m$ 
are elements of $\LL^\times$ satisfying \eqref{eq:mu_compat}. We note that, since $f\bydef\res^{-1}\circ\acute{\beta}$ 
is a homomorphism $K\to\Hc^1(G,\LL^\times)$, the $1$-cocycles $f_{a_i}^{o(a_i)}$ are coboundaries, 
so such $\mu_i$ always exist. Denote 
\begin{equation}\label{eq:modelC_genrel}
\cC(\LL,\theta,\beta,\mu)\bydef
\alg_\LL\langle X_1,\ldots,X_m\mid X_iX_j=\beta_{ij}X_jX_i\text{ and }X_i^{o(a_i)}=\mu_i 1\rangle.
\end{equation}
It is easy to see that assigning degree $a_i$ to the generator $X_i$ makes $\cC(\LL,\theta,\beta,\mu)$ 
a graded-division algebra over $\LL$ with support $K$ and $1$-dimensional homogeneous components 
(see e.g. \cite[Proposition 3.2]{BEK}). Moreover, the center is $\LL$ since $\beta$ is nondegenerate.
If we pull the $\LL$-vector space structure on $\cC\bydef\cC(\LL,\theta,\beta,\mu)$ back along an isomorphism 
$\psi:\LL\to\LL$ from $\Gal(\LL/\FF)$, the resulting algebra $\cC^\psi$ has the same generators $X_1,\ldots,X_m$,
but the relations change: instead of $\mu_i$, we will have $\psi^{-1}(\mu_i)$ (while $\beta_{ij}\in\FF^\times$ stay the 
same). From \eqref{eq:mu_compat} it follows that \eqref{eq:action_on_gen} defines an isomorphism of $\LL$-algebras 
$\sigma_g:\cC^{\theta_g^{-1}}\to\cC$ or, equivalently, a $\theta_g$-semilinear automorphism of $\cC$.
Moreover, since $f_{a_i}:G\to\LL^\times$ is a $1$-cocycle, the mapping $g\mapsto\sigma_g$ is a homomorphism
$G\to\Aut_\FF(\cC)$. Thus $\cC(\LL,\theta,\beta,\mu)$ becomes a $G$-algebra (over $\FF$). Since the homogeneous 
components are the eigenspaces for the action of the subgroup $K$, Corollary \ref{co:main2} tells us that 
$\cC(\LL,\theta,\beta,\mu)$ is a $G$-Galois extension of $\FF$.

\begin{proposition}\label{prop:classification_mu}
The $G$-algebras $\cC(\LL,\theta,\beta,\mu)$ and $\cC(\LL,\theta,\beta',\mu')$ are isomorphic if and only if 
$\beta'=\beta$ and $\mu'_i\in\mu_i(\FF^\times)^{[o(a_i)]}$ for all $i=1,\ldots,m$.
\end{proposition} 

\begin{proof}
Denote the algebras $\cC(\LL,\theta,\beta,\mu)$ and $\cC(\LL,\theta,\beta',\mu')$ by $\cC$ and $\cC'$, 
and their generators (over $\LL$) by $X_i$ and $X_i'$ ($i=1,\ldots,m$), respectively.
If $\beta'=\beta$ and $\mu'_i=\mu_i\lambda_i^{o(a_i)}$ for some $\lambda_i\in\FF^\times$, then the mapping 
$X_i\mapsto\lambda_i^{-1} X'_i$ defines an $\LL$-linear isomorphism $\cC\to\cC'$, which is obviously $G$-equivariant.

Conversely, suppose there is an isomorphism $\psi:\cC\to\cC'$ as $G$-algebras over $\FF$. 
Since $G$ is abelian, each automorphism $\sigma_h$ of $\cC$ is $G$-equivariant. 
Since $\sigma_h$ is $\theta_h$-semilinear and $\theta:G\to\Gal(\LL/\FF)$ is surjective, we can replace $\psi$
by $\psi\circ\sigma_h$ for a suitable $h\in G$ and assume that $\psi$ is $\LL$-linear.
From the commutation relations it follows that $\beta'_{ij}=\beta_{ij}$ for all $i,j$, hence $\beta'=\beta$. 
Since $\psi$ is $G$-equivariant, in particular, it is an isomorphism of $K$-graded algebras, so 
$\psi(X_i)=l_i X'_i$ for some $l_i\in\LL^\times$.
But considering the $G$-action on the generators, we get 
\[
f_{a_i}(g)l_i X'_i=\psi(f_{a_i}(g)X_i)=\psi(g\cdot X_i)=g\cdot\psi(X_i)=g\cdot(l_i X'_i)=\theta_g(l_i)f_{a_i}(g) X'_i,
\]
so $\theta_g(l_i)=l_i$ for all $g\in G$ and hence $l_i\in\FF^\times$. Finally,
\[
\mu_i 1_{\cC'}=\psi(\mu_i 1_{\cC})=\psi(X_i^{o(a_i)})=\psi(X_i)^{o(a_i)}
=l_i^{o(a_i)}(X'_i)^{o(a_i)}=l_i^{o(a_i)}\mu'_i 1_{\cC'},
\]
which implies $\mu'_i\in\mu_i(\FF^\times)^{[o(a_i)]}$.
\end{proof} 

If $\beta$ is fixed, we can use a set of generators of $K$ adapted to $\beta$, namely, a ``symplectic basis''
$a_1,b_1,\ldots,a_m,b_m$ as in \eqref{eq:symplectic_basis}. Then the generators $X_i\bydef X_{a_i}$ and 
$Y_i\bydef X_{b_i}$ of $\cC(\LL,\theta,\beta,\mu)$ satisfy the following defining relations:
\begin{equation*}
\begin{split}
&X_i^{n_i}=\mu_i,\; Y_i^{n_i}=\nu_i,\; X_i Y_i=\zeta_i Y_i X_i,\\
&X_i X_j=X_j X_i,\; Y_i Y_j=Y_j Y_i,\text{ and } X_i Y_j=Y_j X_i\text{ for }i\neq j,
\end{split}
\end{equation*}
where $n_i=o(a_i)=o(b_i)$ and $\zeta_i=\beta(a_i,b_i)$ is a primitive root of unity of degree $n_i$ 
(which can be chosen arbitrarily at the expense of changing the ``symplectic basis'').
This implies that $\cC(\LL,\theta,\beta,\mu)$ is a tensor product of \emph{symbol algebras} (see e.g. \cite[p.~27]{KMRT}):
\begin{equation}\label{eq:symbol_alg}
\cC(\LL,\theta,\beta,\mu)\simeq (\mu_1,\nu_1)_{\zeta_1,\LL}\otimes_\LL\cdots\otimes_\LL (\mu_m,\nu_m)_{\zeta_m,\LL}.
\end{equation}
As seen above, the parameters $\mu_i$ and $\nu_i$ have to satisfy $\dc\mu_i=f_{a_i}^{n_i}$ and $\dc\nu_i=f_{b_i}^{n_i}$
in order to make $\cC(\LL,\theta,\beta,\mu)$ a $G$-Galois extension of $\FF$.


\section{Graded-division algebras}\label{se:gr-div}

Let $G$ be an abelian group. In this section, we classify $G$-graded-division algebras in terms of simpler objects. 

\subsection{Central simple case}
Let $\cD$ be a central simple $G$-graded-division algebra over $\FF$ and let $T$ be the support of $\cD$.
Then $\cD$ represents the class $[\cD]_T$ in $\Br_T(\FF)$ and also the class $[\cD]_G$ in $\Br_G(\FF)$. 
These classes correspond to each other under the canonical embedding of $\Br_T(\FF)$ into $\Br_G(\FF)$ 
(cf. Remark \ref{re:direct_limit}).

If $[\cD]=1$ in $\Br(\FF)$, then applying Corollary \ref{co:PP} with $T$ playing the role of $G$, we get 
$[\cD]_T=[\Gamma(\cD)\# \FF T]_T$, and $\Gamma(\cD)\simeq\cC^\op$ by Proposition \ref{prop:GCD} where 
$\cC=\Cent_\cD(\cD_e)$ with the action of $T$ given by equation \eqref{eq:sigma} (Miyashita-Ulbrich action). 
Hence there is a $T$-graded right $\cD$-module $V$ such that $\cC^\op\# \FF T$ is isomorphic, 
as a $T$-graded algebra, to $\End_\cD(V)$. Therefore, $\cD$ is graded-isomorphic 
to the algebra $E\bigl(\cC^\op\#\FF T\bigr)E$, where $E$ is any primitive idempotent of the identity component 
$\bigl(\cC^\op\#\FF T\bigr)_e=\cC^\op$. By Corollary \ref{co:main} (again, with $T$ playing the role of $G$), 
$\cC\simeq\cC(\LL,\theta,\xi)$ as a $T$-algebra (Definition \ref{df:modelsC}), so it has the form \eqref{eq:symbol_alg}. 
Thus, to recover $\cD$ from $\cC$ explicitly, one needs to find a primitive idempotent in a tensor product of symbol
algebras over $\LL$, which is difficult in general. 

If $[\cD]=[\Delta]$ in $\Br(\FF)$, for a central (ungraded) division algebra $\Delta$ (unique up to isomorphism), then 
Corollary \ref{co:PP} gives $[\cD]_T=[\Delta\otimes_\FF(\cC^\op\# \FF T)]_T$, and hence $\cD$ is recovered, 
up to graded isomorphism, as $E\bigl(\Delta\otimes_\FF(\cC^\op\# \FF T)\bigr)E$, where $E$ is any primitive idempotent 
of $\bigl(\Delta\otimes_\FF(\cC^\op\# \FF T)\bigr)_e=\Delta\otimes_\FF \cC^\op$. 

In particular, Corollary \ref{co:PP} and Proposition \ref{prop:GCD} give the following isomorphism criterion:
\begin{corollary}\label{co:PP2}
Let $G$ be an abelian group and let $\cD$ and $\cD'$ be finite-dimensional $G$-graded-division algebras  
with supports $T$ and $T'$. Assume that $\cD$ and $\cD'$ are central simple over $\FF$. 
Then $\cD$ and $\cD'$ are isomorphic as $G$-graded algebras if and only if 
the following conditions are satisfied:
\begin{romanenumerate}
\item $T=T'$;
\item $\Cent_\cD(\cD_e)$ and $\Cent_{\cD'}(\cD'_e)$ are isomorphic as $T$-algebras;
\item $[\cD]=[\cD']$ in $\Br(\FF)$.\qed
\end{romanenumerate}
\end{corollary}

\begin{remark}\label{re:De}
There is another division algebra associated to $\cD$, namely, the identity component $\cD_e$, which represents an 
element in $\Br(\LL)$ since $\LL=Z(\cD_e)$. This element can be recovered from $\Delta$ and $\cC$ as follows.
We have $\Delta\otimes_\FF(\cC^\op\# \FF T)\simeq\End_\cD(V)$ for a $T$-graded right $\cD$-module $V$.
Since $T$ is the support of $\cD$, $V$ is isomorphic to the direct sum of copies of $\cD$ or, in other words, 
$\FF^k\otimes_\FF\cD$ for some $k$. Hence $\Delta\otimes_\FF(\cC^\op\# \FF T)\simeq\Mat_k(\FF)\otimes_\FF\cD$, 
where the first factor has trivial grading. Looking at the identity components, we obtain 
$\Delta\otimes_\FF \cC^\op\simeq\Mat_k(\FF)\otimes_\FF\cD_e$. Since $\cC$ and $\cD_e$ are $\LL$-algebras,
this can be rewritten as $(\Delta\otimes_\FF\LL)\otimes_\LL\cC^\op\simeq\Mat_k(\LL)\otimes_\LL\cD_e$, so
\[
[\cD_e]=[\Delta\otimes_\FF \LL]\,[\cC]^{-1}\,\text{ in }\Br(\LL).
\]
A stronger result follows from \cite[Theorem 2.3]{BEK}:
\[
[\cD\otimes_\FF \LL]_G=[\cD_e]_G\,[\cC]_G\,\text{ in }\Br_G(\LL).
\]
\end{remark}

\subsection{General case}
Any $G$-graded-simple algebra $\cA$ is $G$-graded-central if considered as an algebra over the field $Z(\cA)_e$. 
Hence we may restrict to $G$-graded-division algebras that are graded-central over $\FF$.
These may be reduced to the central simple case, considered in Corollary \ref{co:PP2}, using a cocycle-twisted 
version of the loop algebra construction introduced in \cite[\S 5]{EldLoop}.

Any $2$-cocycle $\gamma:G\times G\to\FF^\times$ (with trivial action of $G$ on $\FF^\times$) can be used 
to twist the multiplication of a $G$-graded algebra $\cA$ as follows:
\[
x*y\bydef\gamma(g_1,g_2)xy\quad\forall x\in\cA_{g_1}, y\in\cA_{g_2},\, g_1,g_2\in G.
\]
We will denote by $\cA^\gamma$ the vector space $\cA$ equipped with this new multiplication and the original $G$-grading.
The isomorphism class of the $G$-graded algebra $\cA^\gamma$ depends only on the class $[\gamma]\in\Hc^2(G,\FF^\times)$.  

If the support of $\cA$ is contained in a subgroup $H$ of $G$, then the twist of $\cA$ as a $G$-graded algebra by a 
$2$-cocycle in $\Zc^2(G,\FF^\times)$ is the same as the twist of $\cA$ as an $H$-graded algebra by the restriction of 
this $2$-cocycle to $H$. A crucial fact is that the set of graded-isomorphism classes of the twists of $\cA$ by 
\emph{symmetric} $2$-cocycles does not depend on whether we regard $\cA$ as $G$-graded or $H$-graded. The reason is that
any symmetric $2$-cocycle on $H$ can be extended to a symmetric $2$-cocycle on $G$. Indeed, using once again the 
identification of $\Hc^2_{\textup{sym}}$ and $\Ext$ for abelian groups, we have the following exact sequence:
\begin{equation}\label{eq:inf_res_sym}
\begin{split}
1&\longrightarrow \Hom(G/H,\FF^\times)\longrightarrow \Hom(G,\FF^\times)\longrightarrow \Hom(H,\FF^\times)\\
&\longrightarrow \Hc^2_{\textup{sym}}(G/H,\FF^\times)\stackrel{\inf}\longrightarrow 
\Hc^2_{\textup{sym}}(G,\FF^\times)\stackrel{\res}\longrightarrow \Hc^2_{\textup{sym}}(H,\FF^\times)\longrightarrow 1.
\end{split}
\end{equation}
This fact was used in \cite{EldLoop} as follows.
If $\cA$ is a $G$-graded-central-simple algebra (not necessarily finite-dimensional) and 
$H$ is the support of the induced grading on the center(=centroid) $Z(\cA)$, then $Z(\cA)$ is, up to isomorphism, 
a twisted group algebra $\FF^{\tilde{\gamma}}H$ for some $\tilde{\gamma}\in\Zc^2_{\textup{sym}}(H,\FF^\times)$, 
so it can be ``untwisted'' by $\tilde{\gamma}^{-1}$. 
We can find $\gamma\in\Zc^2_{\textup{sym}}(G,\FF^\times)$ whose restriction to $H$ is $\tilde{\gamma}$, and hence
the center of $\cA^{\gamma^{-1}}$ is isomorphic to the group algebra $\FF H$. Following \cite{Allison_et_al}, 
we let $\pi:G\to \overline{G}\bydef G/H$ be the natural homomorphism and consider a \emph{central image} $\bar{\cA}$ of 
$\cA^{\gamma^{-1}}$ induced by an algebra homomorphism from the center to $\FF$. 
By \cite[Theorem 5.2]{EldLoop}, $\bar{\cA}$ is a central simple $\overline{G}$-graded algebra, and $\cA$ is isomorphic, 
as a $G$-graded algebra, to the $\gamma$-twist $L_\pi^\gamma(\bar{\cA})$ of the loop algebra $L_\pi(\bar{\cA})$ 
(see Subsection \ref{sse:intro_loop}).
Conversely, for any central simple $\overline{G}$-graded algebra $\cB$ and any 
$\gamma\in\Zc^2_{\textup{sym}}(G,\FF^\times)$, the cocycle-twisted loop algebra $L_\pi^\gamma(\cB)$ is a 
$G$-graded-central-simple algebra, with $L_\pi^\gamma(\cB)$ and $L_\pi^{\gamma'}(\cB')$ being isomorphic as 
$G$-graded algebras if and only if there exists $\alpha\in\Zc^2_{\textup{sym}}(\overline{G},\FF^\times)$ 
such that $[\gamma'] = \inf([\alpha]^{-1})[\gamma]$ in $\Hc^2_{\textup{sym}}(G,\FF^\times)$ and $\cB'\simeq\cB^\alpha$ as
$\overline{G}$-graded algebras. The first condition determines $[\alpha]$ up to a factor in the image of the connecting 
homomorphism $\Hom(H,\FF^\times)\to\Hc^2_{\textup{sym}}(\overline{G},\FF^\times)$ in \eqref{eq:inf_res_sym}, 
so we still have the freedom to twist $\cB$ by the elements of this image to satisfy the second condition 
(cf. \cite[Corollary~5.5]{EldLoop}).


We apply this procedure to obtain $G$-graded-central-division algebras from central simple 
$\overline{G}$-graded-division algebras. The following result describes the effect of a cocycle twist on the latter 
in terms of the centralizer of the identity component (which carries a natural action of the support of the grading) and 
the class in $\Br(\FF)$.

\begin{lemma}\label{lm:twistD}
Let $\overline{\cD}$ be a central simple $\overline{G}$-graded-division algebra with support $\overline{T}$ 
and let $\alpha\in\Zc^2_{\textup{sym}}(\overline{T},\FF^\times)$. Then the centralizer of the identity component in 
$\overline{\cD}^\alpha$ is the graded vector space $\overline{\cC}\bydef\Cent_{\overline{\cD}}(\overline{\cD}_{\bar{e}})$
with the same $\overline{T}$-action but with multiplication twisted by the restriction of $\alpha$ to the support 
$\overline{K}$ of $\overline{\cC}$. Moreover, if $\alpha$ is the inflation of some 
$\alpha'\in\Zc^2_{\textup{sym}}(\overline{T}/\overline{K},\FF^\times)$ then $[\overline{\cD}^\alpha]\in\Br(\FF)$ is 
the product of $[\overline{\cD}]$ and the element corresponding to the image of $[\alpha']$ 
under the homomorphism $\Hc^2(\overline{T}/\overline{K},\FF^\times)\to\Hc^2(\overline{T}/\overline{K},\LL^\times)$ 
induced by the inclusion of $\FF$ into $\LL\bydef Z(\overline{\cD}_{\bar{e}})$, which is a Galois field extension of $\FF$
with $\Gal(\LL/\FF)\simeq\overline{T}/\overline{K}$. 
\end{lemma}

\begin{proof}
Since the $2$-cocycle $\alpha$ is symmetric, the $\alpha$-twist of multiplication on $\overline{\cD}$ does not affect
commutation relations between homogeneous elements, so the first assertion is clear.

The second assertion can be proved using a variation of the classical argument (see e.g. \cite[Theorem 4.4.3]{Her}) showing 
that multiplication in $\Hc^2(\overline{T}/\overline{K},\LL^\times)$ corresponds to multiplication in $\Br(\FF)$. 
The algebra $\cA\bydef\bigl(\LL\#\FF(\overline{T}/\overline{K})\bigr)^{\alpha'}$ represents the element of $\Br(\FF)$
corresponding to $\alpha'$ and the algebra 
$\tilde{\cA}\bydef\bigl(\LL\#\FF(\overline{T}/\overline{K})\bigr)^{(\alpha')^{-1}}$ 
represents the inverse of this element. 
Let $\Delta$ be a central division algebra that represents $[\overline{\cD}]$. Then 
$\Delta\otimes_\FF(\overline{\cC}^\op\# \FF\overline{T})\simeq\Mat_k(\FF)\otimes_\FF\overline{\cD}$ for some $k$, 
where the first factor has trivial grading (Remark \ref{re:De}). Therefore, it is sufficient to show that 
the algebra $\cB\bydef(\overline{\cC}^\op\# \FF\overline{T})^{\alpha}$ represents the same class as $\cA$ 
in $\Br(\FF)$. Consider $\cE\bydef\tilde{\cA}\otimes_\FF\cB$. Since $\LL\# 1$ is a subalgebra of $\tilde{\cA}$ isomorphic 
to $\LL$ and $\overline{\cC}^\op\# 1$ is a subalgebra of $\cB$ isomorphic to $\overline{\cC}^\op$, 
the algebra $\cE$ contains 
\[
\LL\otimes_\FF\LL\,=\bigoplus_{\varphi\in\Gal(\LL/\FF)}\LL\varepsilon_\varphi,
\]
where $\varepsilon_\varphi$ are orthogonal idempotents satisfying 
$(\varphi_1\otimes\varphi_2)(\varepsilon_\varphi)=\varepsilon_{\varphi_2\varphi\varphi_1^{-1}}$ for all 
$\varphi_1,\varphi_2\in\Gal(\LL/\FF)$. In particular, $\varepsilon\bydef\varepsilon_{\id}$ is a central idempotent 
of the subalgebra $(\LL\# 1)\otimes_\FF(\overline{\cC}^\op\# 1)$, in which it generates an ideal isomorphic to 
$\overline{\cC}^\op$, and we have   
\[ 
\varepsilon\bigl((1\#\FF(\overline{T}/\overline{K}))\otimes_\FF(1\#\FF\overline{T})\bigr)\varepsilon
=\varepsilon S,
\]
where $S$ is the span of all elements of the form $(1\#\bar{t}\overline{K})\otimes (1\#\bar{t})$, $\bar{t}\in\overline{T}$.
But when we multiply two such elements in $\cE$, the values of the cocycles $(\alpha')^{-1}$ and $\alpha$ cancel out, 
so $S$ is a subalgebra of $\cE$ isomorphic to $\FF\overline{T}$. It follows that
\[
\varepsilon\cE\varepsilon\simeq\overline{\cC}^\op\#\FF\overline{T}\simeq\End_\FF(\overline{\cC}^\op),
\]
hence the class of $\cE$ in $\Br(\FF)$ is trivial, proving the result.
\end{proof}

In terms of the isomorphism $\Br_{\overline{T}}(\FF)\to\Br(\FF)\times \uE_{\overline{T}}(\FF)$ 
as in Corollary \ref{co:PP}, sending $[\overline{\cD}]_{\overline{T}}$ to 
$\bigl([\overline{\cD}],[\overline{\cC}^\op]_{\textup{$\overline{T}$-alg}}\bigr)$, 
Lemma \ref{lm:twistD} says that twisting $\overline{\cD}$ by $\alpha$ has the following effect on the image of 
$[\overline{\cD}]_{\overline{T}}$: the second component is twisted by $\res(\alpha)$ and, if $\alpha=\inf(\alpha')$, 
the first component is multiplied by the image of $[\alpha']$ in $\Br(\FF)$. 

By Corollary \ref{co:main}, the isomorphism classes of simple $\overline{T}$-Galois extensions of $\FF$ are parametrized 
by the set $\mathfrak{T}_\FF(\overline{T})$, with a triple 
$(\LL,\theta,\tau\Bc^2(\overline{K},\FF^\times))\in\mathfrak{T}_\FF(\overline{T})$ corresponding to 
the isomorphism class of the $\overline{T}$-algebra $\cC(\LL,\theta,\tau)=\LL^\tau\overline{K}$ 
as in Definition \ref{df:modelsC} (with $\overline{T}$ playing the role of $G$ and $\overline{K}\bydef\ker\theta$).
Clearly, twisting $\cC(\LL,\theta,\tau)$ by a symmetric $2$-cocycle $\overline{K}\times\overline{K}\to\FF^\times$ 
multiplies $\tau$ by this cocycle, hence the isomorphism classes of these twists are parametrized by 
a coset of the subgroup $\Hc^2_{\textup{sym}}(\overline{K},\FF^\times)$ in 
$\Zc^2(\overline{K},\LL^\times)/\Bc^2(\overline{K},\FF^\times)$. As explained after Definition \ref{df:modelsC},  
the elements of such a coset represent the isomorphism classes of simple $\overline{T}$-Galois extensions of $\FF$ 
whose center is isomorphic to $(\LL,\theta)$ and commutation relations are given by a fixed 
nondegenerate alternating bicharacter $\bar{\beta}:\overline{K}\times\overline{K}\to\FF^\times$. 
Thus, $\Hc^2_{\textup{sym}}(\overline{K},\FF^\times)$ acts simply transitively on these isomorphism classes 
by the twist of multiplication. Moreover, up to isomorphism and symmetric cocycle twist, 
the simple $\overline{T}$-Galois extensions of $\FF$ are classified by the triples $(\LL,\theta,\bar{\beta})$. 

Now, for any $G$-graded-central-division algebra $\cD$ with support $T$ and support of the center $H$, 
we have $\cD\simeq L_\pi^\gamma(\overline{\cD})$ where $\overline{\cD}$ is a central simple 
$\overline{G}$-graded-division algebra with support $\overline{T}\bydef T/H$. Moreover, $\overline{\cD}$ 
can be replaced by any of its symmetric cocycle twists at the 
expense of changing $\gamma$.  
In view of Lemma \ref{lm:twistD}, the simple $\overline{T}$-Galois extension of $\FF$ corresponding to $\overline{\cD}$ 
can be replaced by any of its symmetric cocycle twists. Therefore, it makes sense to fix, for any triple  
$(\LL,\theta,\bar{\beta})$, a $2$-cocycle $\tau_0\in\Zc^2(\overline{K},\LL^\times)$ that satisfies 
$\bar{\beta}(\bar{k}_1,\bar{k}_2)=\tau_0(\bar{k}_1,\bar{k}_2)\tau_0(\bar{k}_2,\bar{k}_1)^{-1}$ and 
$f_{\bar{k}_1}f_{\bar{k}_2} = \dc\bigl(\tau_0(\bar{k}_1,\bar{k}_2)\bigr)f_{\bar{k}_1\bar{k}_2}$ for all 
$\bar{k}_1,\bar{k}_2\in \overline{K}$ (which exists by Lemma \ref{lm:tau_exists} and can be obtained explicitly using  
generators $a_i$ and scalars $\mu_i$ as in Definition \ref{df:modelsD} below). 
Then we let $\overline{\cC}\bydef\cC(\LL,\theta,\tau_0)$ as in Definition \ref{df:modelsC} and, 
for any central division algebra $\Delta$ over $\FF$, let $\cD(\overline{T},\LL,\theta,\bar{\beta},\Delta)$ 
be a central simple graded-division algebra whose isomorphism class corresponds to the pair 
$\bigl([\Delta], [\overline{\cC}^\op]_\textup{$\overline{T}$-alg}\bigr)\in \Br(\FF)\times \uE_{\overline{T}}(\FF)$. 
We will now give a more explicit construction of such an algebra:  

\begin{df}\label{df:modelsD}
For any finite subgroup $\overline{T}$ of $\overline{G}$ and any $(\LL,\theta)\in\mathfrak{Z}_\FF(\overline{T})$, 
write $\overline{K}\bydef\ker\theta$ as a direct product of cyclic subgroups generated by $\bar{a}_1,\ldots, \bar{a}_m$. 
Then, for any nondegenerate alternating bicharacter $\bar{\beta}:\overline{K}\times\overline{K}\to\FF^\times$, 
fix scalars $\mu_i\in\LL^\times$ satisfying $\dc\mu_i=f_{\bar{a}_i}^{o(\bar{a}_i)}$ and define a $\theta$-semilinear 
$\overline{T}$-action on the algebra $\overline{\cC}\bydef\cC(\LL,\theta,\bar{\beta},\mu)$ as in \eqref{eq:modelC_genrel},
i.e.,  
\[
\overline{\cC}=\alg_\LL\langle X_1,\ldots,X_m\mid X_iX_j=\bar{\beta}(\bar{a}_i,\bar{a}_j)X_jX_i
\text{ and }X_i^{o(\bar{a}_i)}=\mu_i 1\rangle,
\]
by setting $\bar{t}\cdot X_i=f_{\bar{a}_i}(\overline{t})X_i$ for all $\bar{t}\in\overline{T}$.
Finally, for any finite-dimensional central division algebra $\Delta$ over $\FF$, pick a primitive idempotent 
$E\in\Delta\otimes_\FF\overline{\cC}^\op$ and define
\[
\cD(\overline{T},\LL,\theta,\bar{\beta},\Delta)\bydef E\bigl(\Delta\otimes_\FF(\overline{\cC}^\op\# \FF\overline{T})\bigr)E
=\bigoplus_{\bar{t}\in\overline{T}}E\bigl(\Delta\otimes_\FF(\overline{\cC}^\op\# \bar{t})\bigr)E.
\]
This is a central simple $\overline{G}$-graded-division algebra (with support $\overline{T}$),
whose graded-isomorphism class does not depend on $E$, but depends on our (fixed) choice of generators $a_i$
and scalars $\mu_i$, $i=1,\ldots,m$.
\end{df}

\begin{theorem}\label{th:loop}
Let $G$ be an abelian group and let $\FF$ be a field.
\begin{enumerate}
\item[(1)] If $H$ is a finite subgroup of $G$, $\pi:G\to\overline{G}\bydef G/H$ is the natural homomorphism, 
$\overline{\cD}$ is a finite-dimensional central simple $\overline{G}$-graded-division algebra over $\FF$, 
and $\gamma:G\times G\to\FF^\times$ is a symmetric $2$-cocycle (with trivial action of $G$ on $\FF^\times$), 
then $L_\pi^\gamma(\overline{\cD})$ is a finite-dimensional $G$-graded-central-division algebra over $\FF$.
\item[(2)] If $\cD$ is a finite-dimensional $G$-graded-central-division algebra over $\FF$, $T$ is the support of $\cD$, 
and $H$ is the support of the center of $\cD$, then $\cD$ is graded-isomorphic to some 
$L_\pi^\gamma(\overline{\cD})$ where $\pi$ and $\gamma$ are as above and 
$\overline{\cD}=\cD(T/H,\LL,\theta,\bar{\beta},\Delta)$ is a central simple $\overline{G}$-graded-division algebra as in 
Definition \ref{df:modelsD}.
\item[(3)] Two algebras $L_\pi^\gamma\bigl(\cD(T/H,\LL,\theta,\bar{\beta},\Delta)\bigr)$ and 
$L_{\pi'}^{\gamma'}\bigl(\cD(T'/H',\LL',\theta',\bar{\beta}',\Delta')\bigr)$ as above are graded-isomorphic if and only if 
the following conditions hold:
\begin{enumerate}
\item[(i)] $T=T'$, $H=H'$, $\LL=\LL'$, $\theta=\theta'$, $\bar{\beta}=\bar{\beta}'$; 
\item[(ii)] There exists $\tilde{\alpha}\in\Zc^2_{\textup{sym}}(T/K,\FF^\times)$, where $K\bydef\pi^{-1}(\ker\theta)$, 
such that
\begin{itemize}
\item $\res([\gamma']) = \inf([\tilde{\alpha}]^{-1})\res([\gamma])$ in $\Hc^2_{\textup{sym}}(T,\FF^\times)$;
\item The class $[\Delta']$ in $\Br(\FF)$ is the product of $[\Delta]$ and the image of $[\tilde{\alpha}]$ 
under the homomorphism $\Hc^2_{\textup{sym}}(T/K,\FF^\times)\to\Hc^2(T/K,\LL^\times)$ induced by the inclusion $\FF\to\LL$,
where $\Hc^2(T/K,\LL^\times)$ is regarded as a subgroup of $\Br(\FF)$ using the identification 
$T/K\simeq\Gal(\LL/\FF)$ via $\theta$.
\end{itemize}
\end{enumerate}
\end{enumerate}
\end{theorem}

\begin{proof}
(1) is clear from \cite[Theorem 5.2]{EldLoop} and (2) has already been proved. 

To prove (3), let  $\overline{\cD}=\cD(T/H,\LL,\theta,\bar{\beta},\Delta)$ and 
$\overline{\cD}'=\cD(T'/H',\LL',\theta',\bar{\beta}',\Delta')$.
First note that conditions $T=T'$ and $H=H'$ are necessary for $L_\pi^\gamma(\overline{\cD})$ and 
$L_{\pi'}^{\gamma'}(\overline{\cD}')$ to be graded-isomorphic, so we may consider these algebras as $T$-graded 
and apply \cite[Theorem 5.2]{EldLoop} with $T$ playing the role of $G$. 
Then the remaining condition for graded isomorphism is the existence of 
$\alpha\in\Zc^2_{\textup{sym}}(\overline{T},\FF^\times)$ such that $\res([\gamma']) = \inf([\alpha]^{-1})\res([\gamma])$ in 
$\Hc^2_{\textup{sym}}(T,\FF^\times)$ and $\overline{\cD}'\simeq\overline{\cD}^\alpha$ as $\overline{T}$-graded algebras,
where $\overline{T}\bydef T/H$.
We have already seen that $\overline{\cD}'\simeq\overline{\cD}^\alpha$ implies $\LL'=\LL$, $\theta'=\theta$ and 
$\bar{\beta}'=\bar{\beta}$. Moreover, the restriction of $\alpha$ to $\overline{K}\bydef K/H$ must give the trivial element
of $\Hc^2_{\textup{sym}}(\overline{K},\FF^\times)$, since the graded-division algebras in Definition \ref{df:modelsD}
were constructed from representatives of orbits for the free action of $\Hc^2_{\textup{sym}}(\overline{K},\FF^\times)$ on 
the isomorphism classes of simple $\overline{T}$-Galois extensions of $\FF$.
Now, since $\Hc^2_{\textup{sym}}(\,\cdot\,,\FF^\times)$ is a functor, we have the following commutative diagram:
\[
\begin{tikzcd}
\Hc^2_{\textup{sym}}(\overline{T}/\overline{K},\FF^\times)\arrow[r, "\inf"]\arrow[d, "\sim"]
&\Hc^2_{\textup{sym}}(\overline{T},\FF^\times)\arrow[r, "\res"]\arrow[d, "\inf"]
&\Hc^2_{\textup{sym}}(\overline{K},\FF^\times)\arrow[r]\arrow[d, "\inf"]&1\\
\Hc^2_{\textup{sym}}(T/K,\FF^\times)\arrow[r, "\inf"]
&\Hc^2_{\textup{sym}}(T,\FF^\times)\arrow[r, "\res"]
&\Hc^2_{\textup{sym}}(K,\FF^\times)\arrow[r]&1
\end{tikzcd}
\]
where the rows are exact because they are the second halves of sequences similar to \eqref{eq:inf_res_sym}. 
Hence we have $[\alpha]=\inf([\alpha'])$ for some $\alpha'\in\Zc^2_{\textup{sym}}(\overline{T}/\overline{K},\FF^\times)$.
Let $\tilde{\alpha}$ be the element of $\Zc^2_{\textup{sym}}(T/K,\FF^\times)$ corresponding to $\alpha'$. Then, 
in view of the above diagram, equation $\res([\gamma']) = \inf([\alpha]^{-1})\res([\gamma])$ in 
$\Hc^2_{\textup{sym}}(T,\FF^\times)$ is equivalent to the first part of (ii), 
while, in view of Corollary \ref{co:PP2} and Lemma \ref{lm:twistD}, condition 
$\overline{\cD}'\simeq\overline{\cD}^\alpha$ is equivalent to the second part of (ii).
\end{proof}

\begin{corollary}\label{co:loop}
The set of isomorphism classes of finite-dimensional $G$-graded-central-division algebras over $\FF$ is in bijection
with the following set of septuples $(T,H,\LL,\theta,\bar{\beta},\delta,[\eta])$:
\begin{itemize}
\item $H\le T$ are finite subgroups of $G$; 
\item $\LL$ is a finite Galois extension of $\FF$ contained in $\overline{\FF}$ and $\theta$ is an epimorphism 
$T/H\to\Gal(\LL/\FF)$ such that $|\Hc^1(T/H,\LL^\times)|=|K/H|$ where $K/H$ is the kernel of $\theta$
(so $\theta$ yields an isomorphism $T/K\to\Gal(\LL/\FF)$);
\item $\bar{\beta}$ is a nondegenerate alternating bicharacter on $K/H$ with values in $\FF^\times$;
\item $\delta$ is a coset in $\Br(\FF)$ of the image of $\Hom(K,\FF^\times)$ under the composition of the 
connecting homomorphism $\Hom(K,\FF^\times)\to\Hc^2_{\textup{sym}}(T/K,\FF^\times)$ and the homomorphism
$\Hc^2_{\textup{sym}}(T/K,\FF^\times)\to\Hc^2(T/K,\LL^\times)$ induced by the inclusion $\FF\to\LL$;
\item $\eta\in\Zc^2_{\textup{sym}}(K,\FF^\times)$.
\end{itemize}
\end{corollary}

\begin{proof}
In the notation of Theorem \ref{th:loop}, we let $\delta$ be the coset containing $[\Delta]$ and let 
$\eta$ be the restriction of $\gamma$ to $K$.
Observe that part (3) gives us a lot of freedom in the choice of $\gamma$ for a given $\cD$. 
Indeed, $\gamma$ and $\gamma'$ satisfy the first part of condition (ii) for some $\tilde{\alpha}$ if and only if 
$\res_K([\gamma])=\res_K([\gamma'])$, which means that we can replace $\gamma$ by an arbitrary extension of $\eta$. 
So we fix, for each subgroup $K$ of $G$, a set-theoretic section
\[
\xi_K:\Hc^2_{\textup{sym}}(K,\FF^\times)\rightarrow\Hc^2_{\textup{sym}}(G,\FF^\times)
\]
of the restriction homomorphism and stipulate that $\gamma$ for realizing $\cD$ as $L_\pi^\gamma(\overline{\cD})$ 
be chosen according to this section: $[\gamma]=\xi_K\bigl(\res_K([\gamma])\bigr)$. Under this stipulation, $\tilde{\alpha}$
in condition (ii) must belong to the kernel of 
$\inf:\Hc^2_{\textup{sym}}(T/K,\FF^\times)\to\Hc^2_{\textup{sym}}(T,\FF^\times)$, which is the image of the 
connecting homomorphism above.
\end{proof}

The parameters corresponding to the isomorphism class of a graded-central-division algebra $\cD$ in 
Corollary \ref{co:loop} have the following meaning: $T$ is the support of $\cD$, $H$ is the support of $Z(\cD)$,
$\LL\simeq Z(\cD_e)$, $\theta$ is induced by the epimorphism $\bar{\sigma}:T\to\Gal(\LL/\FF)$ 
(see Proposition \ref{pr:DCK}, with $T$ playing the role of $G$), $K$ is the kernel of $\bar{\sigma}$ and the 
support of $\cC\bydef\Cent_\cD(\cD_e)$, and $\bar{\beta}$ is induced by the alternating bicharacter 
$\beta:K\times K\to\FF^\times$ defined by the commutation relations in $\cC$ (see Proposition \ref{pr:Kbeta}).
The meaning of $\delta$ and $[\eta]$ is less direct and depends on the choices we made in Definition \ref{df:modelsD} and 
the proof of Corollary \ref{co:loop}: $\eta$ is a symmetric $2$-cocycle such that $\cC\simeq L_\pi^\eta(\overline{\cC})$ 
where $\overline{\cC}\bydef\cC(\LL,\theta,\bar{\beta},\mu)$, and $\delta$ is the coset containing the element 
$[\overline{\cD}]\in\Br(\FF)$ where $\overline{\cD}$ is a central image of $\cD^{\gamma^{-1}}$ and 
$[\gamma]=\xi_K([\eta])$.

Finally, we note that all data in Corollary \ref{co:loop} involving symmetric $2$-cocycles can be computed explicitly. 
For the connecting homomorphism $\Hom(K,\FF^\times)\to\Hc^2_{\textup{sym}}(T/K,\FF^\times)$, 
take a set-theoretic section $s_K:T/K\to T$ of the natural homomorphism $T\to T/K$ and, 
for any $\lambda\in\Hom(K,\FF^\times)$, define 
\begin{equation*}
\alpha_\lambda(x_1,x_2)\bydef \lambda\bigl(s_H(x_1)s_H(x_2)s_H(x_1 x_2)^{-1}\bigr)\quad
\forall x_1,x_2\in T/K.
\end{equation*}
Then the class $[\alpha_\lambda]\in\Hc^2_{\textup{sym}}(T/K,\FF^\times)$ is the image of $\lambda$ under 
the connecting homomorphism \cite[III, Lemma 1.4 and Theorem 9.1]{MacLane}.

As to symmetric $2$-cocycles $K\times K\to\FF^\times$ and their extensions, they can be found using the short exact sequence 
$1\to\FF^\times\to\overline{\FF}^\times\to\overline{\FF}^\times/\FF^\times\to 1$ as an injective resolution of $\FF^\times$.
Hence $\Hc^2_{\textup{sym}}(K,\FF^\times)$ is isomorphic to the quotient of 
$\Hom(K,\overline{\FF}^\times\hspace{-3pt}/\FF^\times)$ by the image of $\Hom(K,\overline{\FF}^\times)$. 
Under this isomorphism, the coset of $\chi:K\to\overline{\FF}^\times\hspace{-3pt}/\FF^\times$ corresponds 
to the class of the following symmetric $2$-cocycle:
\begin{equation}\label{eq:Z2sym_from_inj_res}
\eta_\chi(k_1,k_2)\bydef s_\FF(\chi(k_1))s_\FF(\chi(k_2))s_\FF(\chi(k_1 k_2))^{-1}\quad
\forall k_1,k_2\in K,
\end{equation}
where $s_\FF$ is a set-theoretic section of the natural homomorphism 
$\overline{\FF}^\times\to\overline{\FF}^\times\hspace{-3pt}/\FF^\times$. Extending $\chi$ to a homomorphism 
$\tilde{\chi}:G\to\overline{\FF}^\times\hspace{-3pt}/\FF^\times$, we obtain an extension 
$\gamma\bydef\eta_{\tilde{\chi}}$ of the $2$-cocycle $\eta_\chi$. 

\begin{remark}\label{re:twisted_loop_as_form}
We can realize $L_\pi^\gamma(\overline{\cD})$ as an $\FF$-form of 
the $G$-graded algebra $L_\pi(\overline{\cD})\otimes_\FF\overline{\FF}$ as follows. Setting 
$z_g\bydef s_\FF(\tilde{\chi}(g))\in\overline{\FF}^\times$ for all $g\in G$, we have
\[
L_\pi^\gamma(\overline{\cD})\simeq\bigoplus_{g\in G}\overline{\cD}_{\pi(g)}\otimes g\otimes z_g\subset
\overline{\cD}\otimes_\FF\overline{\FF}G,
\]
where we have identified $\FF G\otimes_\FF\overline{\FF}$ with $\overline{\FF}G$. 
(Cf. \cite[Proposition 3.5(iii)]{EldLoop}.)
\end{remark}

\subsection{Finite graded-division rings}

As an application of the method developed in this section, we can classify finite graded-division rings. 
Let $\cD=\bigoplus_{g\in G}\cD_g$ be a finite graded-division ring where $G$ is an abelian group. 
Then $\cD$ can be considered as a graded algebra over the finite field $\FF\bydef Z(\cD)_e$.
When viewed in this way, $\cD$ is a finite-dimensional graded-central-division algebra, so Theorem \ref{th:loop} applies.
Our classification will be up to graded-isomorphism, so we will fix $G$. Clearly, we can also fix the isomorphism class of
$\FF$ and consider finite-dimensional graded-central-division algebras over $\FF$. 
Note, however, that two such objects $\cD$ and $\cD'$ may be isomorphic as graded rings without being isomorphic
as graded $\FF$-algebras. More precisely, $\cD$ is isomorphic to  $\cD'$ as a graded ring if and only if there exists
an automorphism $\psi$ of $\FF$ such that $\cD^\psi$ (the result of pulling scalar multiplication back along $\psi$) is 
isomorphic to $\cD'$ as a graded $\FF$-algebra. 

It will be sometimes convenient to assume that $G$ is the support of the grading, hence finite.
Since $\FF$ is also finite, there is a number of simplifications. First, $\Br(\FF)$ is trivial and, hence, 
the mapping $\cD\mapsto\Cent_\cD(\cD_e)$ yields a bijection between the isomorphism
classes of central simple graded-division $\FF$-algebras with support $G$ and the isomorphism classes 
of simple $G$-Galois extensions of $\FF$. (In terms of Corollary \ref{co:loop}, the parameter $\delta$ is trivial.)
Second, the multiplicative groups of finite fields and the Galois groups of their finite extensions are cyclic,
which allows explicit computations.

Denote the characteristic of $\FF$ by $p$ and let $GF(p^\infty)$ be an algebraic closure of the prime field $GF(p)$.
(It can be constructed as a direct limit of the fields $GF(p^k)$, $k\in\NN$, hence our notation.) 
We may assume that $GF(p)\subset\FF\subset GF(p^\infty)$. It is well known and easy to prove that the multiplicative group
$GF(p^\infty)^\times$ is isomorphic to a subgroup of the unit circle $U\subset\CC$, namely, 
the direct sum of $U_{q^\infty}$ over all primes $q\ne p$, where $U_{q^\infty}$ denotes the group of all complex roots 
of unity of degrees $q^k$, $k\in\NN$. We will fix such an isomorphism and, for any $N\in\NN$ with $p\nmid N$, 
let $\omega_N$ be the element of $GF(p^\infty)^\times$ corresponding to $\exp\frac{2\pi\boldsymbol{i}}{N}$.
Thus, $\omega_N$ is a primitive $N$-th root of unity in $GF(p^\infty)$ and we have $\omega_N^d=\omega_{N/d}$ for any 
$d\mid N$. It follows that the multiplicative group $\FF^\times$ is generated by $\omega_{|\FF^\times|}$, which we will 
abbreviate as $\omega_\FF$.

Denote by $\varphi$ the Frobenius automorphism of $GF(p^\infty)$: $\varphi(x)=x^p$. By abuse of notation, we will use 
the same letter for the restrictions of $\varphi$ to finite subfields of $GF(p^\infty)$. Thus, for any such subfield 
$\LL$ containing $\FF$, the Galois group $\Gal(\LL/\FF)$ is generated by $\varphi^e$ where $|\FF|=p^e$. It follows that 
the norm $N_{\LL/\FF}:\LL\to\FF$ is given by $N_{\LL/\FF}(x)=x^{[\LL^\times:\FF^\times]}$. Hence the mapping 
$\omega_\FF^j\mapsto\omega_\LL^j$ ($0\le j<|\FF^\times|$) is a set-theoretic section of the group epimorphism 
$N_{\LL/\FF}:\LL^\times\to\FF^\times$.

The above information gives us an explicit description of the set $\mathfrak{Z}_\FF(G)$: 
it is in bijection with the cosets $C$ of the subgroups $K\le G$ satisfying the following conditions: 
$G/K$ is a cyclic group generated by $C$, $K\simeq A\times A$ for some abelian group $A$, 
and $\exp(K)$ divides $|\FF^\times|=p^e-1$.
Indeed, $\LL$ is determined by the degree $n\bydef[\LL:\FF]=[G:K]$ and $\theta:G\to\Gal(\LL/\FF)$ is determined 
by its kernel $K$ and the stipulation that it maps the elements of $C$ to $\varphi^e$. 
Next, we have to fix an extension of every character $\lambda:K\to\FF^\times$ to a $1$-cocycle $G\to\LL^\times$.
These extensions exist because $\Br(\LL)=1$ and hence the ``transgression'' map is trivial. Explicitly, they are given
by the following lemma, whose proof is left to the reader:

\begin{lemma}\label{lm:extend_to_1cocycles}
Suppose $\theta(t_0)=\varphi^e$ and $\lambda\in\Hom(K,\FF^\times)$. Then the extensions of $\lambda$ to $1$-cocycles
$G\to\LL^\times$ are as follows: for any $\mu_0\in\LL^\times$ satisfying $N_{\LL/\FF}(\mu_0)=\lambda(t_0^n)$, there 
exists a unique extension that sends $t_0$ to $\mu_0$.\qed
\end{lemma}

For any coset $C$ of $K$ that generates $G/K$, we fix a representative $t_0=t_0(C)$. 
Then we extend every $\lambda\in\Hom(K,\FF^\times)$ as in Lemma \ref{lm:extend_to_1cocycles} with 
$\mu_0\bydef\omega_\LL^j$ where $\lambda(t_0^n)=\omega_\FF^j$ and $0\le j<|\FF^\times|$. 
Thus, given a nondegenerate alternating bicharacter $\beta:K\times K\to\FF^\times$,
the extension $f_k$ of the character $\beta(\cdot,k)$ is determined by the condition 
$f_k(t_0(C))=\omega_\LL^{j_{C,\beta}(k)}$ where, for any $k\in K$, we define $j_{C,\beta}(k)$ by 
\begin{equation}\label{eq:jCb}
\beta(t_0(C)^n,k)=\omega_\FF^{j_{C,\beta}(k)}\text{ and }0\le j_{C,\beta}(k)<|\FF^\times|.
\end{equation} 
This determines the objects $\cC(\LL,\theta,\tau)$ as in Definition \ref{df:modelsC}. 
Explicitly, we can construct them in terms of generators and relations \eqref{eq:modelC_genrel} if we write $K$ 
as a direct product of cyclic subgroups and fix their generators:

\begin{proposition}\label{pr:Galois_ext_of_finite_field}
Let $G$ be a finite abelian group and let $\FF=GF(p^e)$. 
Then the simple $G$-Galois extensions of $\FF$ are classified as follows: 
for any coset $C$ of a subgroup $K\le G$ such that $C$ generates $G/K$ and any nondegenerate alternating bicharacter 
$\beta:K\times K\to\FF^\times$, there are exactly $|K|$ isomorphism classes.
Moreover, if we write $K=\langle a_1\rangle\times\cdots\times\langle a_m\rangle$, the following are representatives
of these isomorphism classes:
\[
\cC(s_1,\ldots,s_m)\bydef
\alg_\LL\langle X_1,\ldots,X_m\mid X_iX_j=\beta(a_i,a_j)X_jX_i\text{ and }X_i^{o(a_i)}=\mu_i(s_i) 1\rangle,
\]
where each $s_i$ is an integer in the interval $0\le s_i<o(a_i)$, 
$\LL$ is the field extension of $\FF$ of degree $n=[G:K]$ contained in the algebraic closure $GF(p^\infty)$,
the elements of the coset $C^j$ act on $\LL$ as $\varphi^{ej}$, 
\[
\mu_i(s_i)\bydef\omega_\LL^{o(a_i)j_{C,\beta}(a_i)/|\FF^\times|}\omega_\FF^{s_i}
\text{ with }j_{C,\beta}:K\to\ZZ_{\ge 0}\text{ defined by \eqref{eq:jCb}},
\]
and the (semilinear) $G$-action on $\cC(s_1,\ldots,s_m)$ is defined by $g\cdot X_i=f_{a_i}(g)X_i$ for all $g\in G$,
$i=1,\ldots,m$.
\end{proposition}

\begin{proof}
Recall that the elements $\mu_i\in\LL^\times$ in \eqref{eq:modelC_genrel} are subject to the condition 
$\dc\mu_i=f_{a_i}^{o(a_i)}$. Since $1$-coboundaries are trivial on $K$, it is sufficient to check it 
for $g=t_0(C)$, which is $\varphi^e(\mu_i)\mu_i^{-1}=\omega_\LL^{o(a_i)j_{C,\beta}(a_i)}$.
Writing $\mu_i=\omega_\LL^{y_i}$, this equation becomes $|\FF^\times|y_i\equiv o(a_i)j_{C,\beta}(a_i)\pmod{|\LL^\times|}$.
Since $\beta(t_0(C)^n,a_i)^{o(a_i)}=1$, we have $o(a_i)j_{C,\beta}(a_i)\equiv 0\pmod{|\FF^\times|}$, so the set of 
all possible values of $y_i$ is given by 
\[
y_i=\frac{o(a_i)j_{C,\beta}(a_i)}{|\FF^\times|}+[\LL^\times:\FF^\times]s_i\text{ where }s_i\in\ZZ.
\]
This gives the desired expression for $\mu_i$ because $\omega_\LL^{[\LL^\times:\FF^\times]}=\omega_\FF$.
The result now follows from Proposition \ref{prop:classification_mu} since the group $(\FF^\times)^{[o(a_i)]}$ is 
generated by $\omega_\FF^{o(a_i)}$.
\end{proof}

This result also gives a classification of central simple graded-division algebras over $\FF$ with support $G$: for each 
$\cC(s_1,\ldots,s_m)$ as in Proposition \ref{pr:Galois_ext_of_finite_field}, there is a unique (up to graded-isomorphism)
such graded-division algebra $\cD(s_1,\ldots,s_m)$ for which the centralizer of the identity component is isomorphic to
$\cC(s_1,\ldots,s_m)$ as a $G$-algebra. Since all the $G$-algebras $\cC(s_1,\ldots,s_m)$ are symmetric cocycle 
twists of each other, the same is true for the $G$-graded algebras $\cD(s_1,\ldots,s_m)$ by Lemma \ref{lm:twistD}. 
So, we can take $\cD(0,\ldots,0)$ as a ``base point'' and denote it by $\cD(G,C,\beta)$ (cf. Definition \ref{df:modelsD}). 
Specializing Corollary \ref{co:loop}, we obtain:

\begin{corollary}\label{co:loop_finite_field}
Let $G$ be an abelian group and let $\FF=GF(p^e)$. 
The set of isomorphism classes of finite-dimensional $G$-graded-central-division algebras over $\FF$ is in bijection
with the following set of quintuples $(T,H,C,\bar{\beta},[\eta])$:
\begin{itemize}
\item $H\le T$ are finite subgroups of $G$; 
\item $C$ is a coset of a subgroup $K$ of $T$ such that $H\le K$ and $C$ generates $T/K$;
\item $\bar{\beta}$ is a nondegenerate alternating bicharacter on $K/H$ with values in $\FF^\times$;
\item $\eta\in\Zc^2_{\textup{sym}}(K,\FF^\times)$.
\end{itemize}
Moreover, $L_\pi^\gamma(\cD(T/H,C/H,\bar{\beta}))$ is a representative of the isomorphism class with parameters 
$(T,H,C,\bar{\beta},[\eta])$, where $\pi:G\to G/H$ is the natural homomorphism and 
$\gamma\in\Zc^2_{\textup{sym}}(G,\FF^\times)$ is any extension of $\eta$.
\qed
\end{corollary}

To address the problem of isomorphism as graded rings, we will use the following:

\begin{lemma}\label{lm:Frobenius_twist}
If $\psi\in\Aut(\FF)$ then the graded $\FF$-algebra $\cD(G,C,\beta)^{\psi^{-1}}$ is isomorphic to a symmetric cocycle twist
of $\cD(G,C,\psi\circ\beta)$. Moreover, if $\psi\circ\beta=\beta$ then $\cD(G,C,\beta)^{\psi^{-1}}$ is isomorphic to 
$\cD(G,C,\beta)$ as a graded $\FF$-algebra.
\end{lemma}

\begin{proof}
Denote $\cD=\cD(G,C,\beta)$. Since the centralizer of the identity component in $\cD$ is $\cC\bydef\cC(0,\ldots,0)$ 
as in Proposition \ref{pr:Galois_ext_of_finite_field}, the centralizer of the identity component in 
$\cD^{\psi^{-1}}$ is $\cC^{\psi^{-1}}$. So, it suffices to consider the latter. 
By construction, $\cC$ is an $\LL$-algebra with a semilinear $G$-action. 
Extending $\psi$ to an automorphism of $\LL$, we can pull the $\LL$-vector space 
structure back along $\psi^{-1}$, which makes $\cC^{\psi^{-1}}$ an $\LL$-algebra with a semilinear $G$-action
of the same kind as $\cC$ since $\psi$ commutes with the elements of $\Gal(\LL/\FF)$.
Now, $\cC^{\psi^{-1}}$ has the same generators $X_1,\ldots,X_m$, but the relations change: $\beta(a_i,a_j)$ is 
replaced by $\psi(\beta(a_i,a_j))$ and $\mu_i(0)$ is replaced by $\psi(\mu_i(0))$. This proves the first assertion.

For the second assertion, let $\FF_0$ be the subfield of $\FF$ generated by the values of $\beta$ or, in other words, by 
$\omega_{\exp(K)}$. Then $\psi\circ\beta=\beta$ if and only if $\psi$ is a power of $\varphi^{e_0}$ where $|\FF_0|=p^{e_0}$.
So, it suffices to consider the case $\psi=\varphi^{e_0}$. The $G$-action on the generators of $\cC^{\psi^{-1}}$
is given by $g\cdot X_i=\psi(f_{a_i}(g))X_i$ (with multiplication by 
scalars taken in $\cC^{\psi^{-1}}$!). Since $\psi\circ\beta=\beta$, we have $\psi(f_{a_i}(g))=f_{a_i}(g)$ for 
$g\in K$, but for $t_0=t_0(C)$, we have $\psi(f_{a_i}(t_0))=\psi(\omega_\LL^{j_i})=\omega_\LL^{p^{e_0}j_i}
=\omega_\LL^{|\FF_0^\times|j_i}f_{a_i}(t_0)$ where $j_i\bydef j_{C,\beta}(a_i)$. 
From the definition of $j_i$ and the fact that $\psi\circ\beta=\beta$ it follows that 
$p^{e_0}j_i\equiv j_i\pmod{|\FF^\times|}$, so $|\FF^\times|$ is a divisor of $|\FF_0^\times|j_i$. Taking the elements
\[
X'_i\bydef\omega_\LL^{-|\FF_0^\times|j_i/|\FF^\times|}X_i,\, i=1,\ldots,m,
\]
as the new generators of $\cC^{\psi^{-1}}$, we get $g\cdot X'_i=f_{a_i}(g)X'_i$ for all $g\in G$. Also, 
\[
(X'_i)^{o(a_i)}=\omega_\LL^{-o(a_i)|\FF_0^\times|j_i/|\FF^\times|}\psi(\mu_i(0))
=\omega_\LL^{-o(a_i)|\FF_0^\times|j_i/|\FF^\times|}\omega_\LL^{p^{e_0}o(a_i)j_i/|\FF^\times|}=\mu_i(0).
\]
It follows that $\cC^{\psi^{-1}}\simeq\cC$ as a $G$-algebra over $\FF$.
\end{proof}

Lemma \ref{lm:Frobenius_twist} allows us to classify the graded-central-division $\FF$-algebras in 
Corollary \ref{co:loop_finite_field} up to isomorphism of graded rings. 
Indeed, for any $G/H$-graded algebra $\overline{\cA}$, the loop algebra $L_\pi(\overline{\cA})$ 
can be seen as a graded subalgebra of the group ring $\overline{\cA}G$ 
(with the $G$-grading defined by declaring $\overline{\cA}g$ to
be the homogeneous component of degree $g$),
so $L_\pi(\overline{\cA})^\psi\simeq L_\pi(\overline{\cA}^\psi)$ for any $\psi\in\Aut(\FF)$. Also, for any 
$G$-graded algebra $\cA$ and any $\gamma\in\Zc^2(G,\FF^\times)$, we have 
$(\cA^\gamma)^{\psi^{-1}}=(\cA^{\psi^{-1}})^{\psi\circ\gamma}$. 
Therefore, using the notation of Corollary \ref{co:loop_finite_field}, we obtain, for any $\psi\in\Aut(\FF)$,  
$L_\pi^\gamma(\cD(T/H,C/H,\bar{\beta}))^{\psi^{-1}}\simeq L_\pi^{\gamma'}(\cD(T/H,C/H,\psi\circ\bar{\beta}))$  
for some $\gamma'\in\Zc^2_{\textup{sym}}(G,\FF^\times)$ that has the same restriction to $H$ (but not necessarily to $K$)
as $\psi\circ\gamma$. Moreover, if $\psi\circ\bar{\beta}=\bar{\beta}$ then
$L_\pi^\gamma(\cD(T/H,C/H,\bar{\beta}))^{\psi^{-1}}\simeq L_\pi^{\psi\circ\gamma}(\cD(T/H,C/H,\bar{\beta}))$.

It is well known that $\Ext(G,\ZZ/N\ZZ)$ is naturally isomorphic to the dual group of $G_{[N]}$, which gives us 
the group $\Hc^2_{\textup{sym}}(G,\FF^\times)$ since $\FF^\times$ is cyclic.
In fact, symmetric $2$-cocycles representing the elements of $\Hc^2_{\textup{sym}}(G,\FF^\times)$ 
can be constructed explicitly. For any $N\in\NN$ with $p\nmid N$, it is convenient to define the map
\[
[1/N]:GF(p^\infty)^\times\to GF(p^\infty)^\times,\;\omega_M^j\mapsto\omega_{MN}^j\text{ where }0\le j<M.
\]
As the notation suggests, this is a set-theoretic section of the epimorphism 
$[N]:GF(p^\infty)^\times\to GF(p^\infty)^\times$, so we will denote the image of $x$ under $[1/N]$ by $x^{1/N}$.

\begin{lemma}\label{lm:H2sym_finite_field}
Let $G$ be an abelian group, $\FF=GF(p^e)$ and $N=p^e-1$. Then there is a natural isomorphism 
$\Hom(G_{[N]},\FF^\times)\simeq\Hc^2_{\textup{sym}}(G,\FF^\times)$ constructed as follows: given a character 
$\chi:G_{[N]}\to\FF^\times$, if $\tilde{\chi}\in\Hom(G,GF(p^\infty)^\times)$ is an extension of $\chi$ then 
\[
\gamma_{\tilde{\chi}}(g_1,g_2)\bydef
\tilde{\chi}(g_1)^{1/N}\tilde{\chi}(g_2)^{1/N}\bigl(\tilde{\chi}(g_1g_2)^{1/N}\bigr)^{-1}\quad\forall g_1,g_2\in G
\]
is a symmetric $2$-cocycle $G\times G\to\FF^\times$ whose class 
$[\gamma_{\tilde{\chi}}]\in\Hc^2_{\textup{sym}}(G,\FF^\times)$ depends only on $\chi$, and the mapping
$\chi\mapsto[\gamma_{\tilde{\chi}}]$ is the desired isomorphism.
\end{lemma}

\begin{proof}
Let $\overline{\FF}=GF(p^\infty)$. Using the exact sequence 
$1\to\FF^\times\to\overline{\FF}^\times\stackrel{[N]}{\longrightarrow}\overline{\FF}^\times\to 1$
as an injective resolution of $\FF^\times$, we obtain the isomorphism
\[
\Hom(G,\overline{\FF}^\times)/\Hom(G,\overline{\FF}^\times)^{[N]}\to\Hc^2_{\textup{sym}}(G,\FF^\times),
\]
explicitly given by a formula similar to \eqref{eq:Z2sym_from_inj_res}, with $[1/N]$ playing the role of $s_\FF$.
Now, applying the exact functor $\Hom(\,\cdot\,,\overline{\FF}^\times)$ to the exact sequence 
$1\to G_{[N]}\to G\stackrel{[N]}{\longrightarrow} G$, we obtain the isomorphism 
\[
\Hom(G,\overline{\FF}^\times)/\Hom(G,\overline{\FF}^\times)^{[N]}\to\Hom(G_{[N]},\overline{\FF}^\times)
=\Hom(G_{[N]},\FF^\times)
\]
induced by the restriction of characters from $G$ to $G_{[N]}$. The result follows.
\end{proof}

Using Lemma \ref{lm:H2sym_finite_field} and Remark \ref{re:twisted_loop_as_form}, representatives of the isomorphism classes
of graded-central-division algebras over $\FF$ in Corollary \ref{co:loop_finite_field} can be constructed explicitly.
For the purpose of classification up to isomorphism of graded rings, we also fix orbit representatives for bicharacters:

\begin{df}\label{df:modelsD_finite_field}
Let $\FF=GF(p^e)$. 
For any finite subgroups $H\le K\le T$ of $G$ such that $T/K$ is cyclic and $\overline{K}\bydef K/H$ admits 
nondegenerate alternating bicharacters with values in $\FF^\times$ 
(i.e., $\overline{K}\simeq A\times A$ for some abelian group $A$ and $\exp(\overline{K})$ divides $|\FF^\times|$), 
fix a representative $\bar{\beta}_0(\cO)$ for any $\Aut(\FF)$-orbit of these bicharacters and also
fix elements $\bar{a}_1,\ldots,\bar{a}_m\in\overline{K}$ such that 
$\overline{K}=\langle \bar{a}_1\rangle\times\cdots\times\langle \bar{a}_m\rangle$. 
Finally, for any coset $C$ of $K$ in $T$ that generates $T/K$, fix an element $\bar{t}_0(C)$ in the coset 
$\overline{C}\bydef C/H$ of $\overline{K}$ in $\overline{T}\bydef T/H$. 
Let $n=[T:K]$ and define a $\overline{T}$-action on $\LL\bydef GF(p^{en})$ by letting the elements of 
$\overline{C}^j$ act as $\varphi^{ej}$.
Then, for any nondegenerate alternating bicharacter $\bar{\beta}:\overline{K}\times\overline{K}\to\FF^\times$, 
let $\lambda_i=\bar{\beta}(\bar{t}_0(C)^n,\bar{a}_i)$ and define a semilinear $\overline{T}$-action on  
\[
\overline{\cC}\bydef\alg_\LL\langle X_1,\ldots,X_m\mid X_iX_j=\bar{\beta}(\bar{a}_i,\bar{a}_j)X_jX_i
\text{ and }X_i^{o(\bar{a}_i)}=\lambda_i^{\frac{1}{|\LL^\times|/o(\bar{a}_i)}} 1\rangle
\]
by setting $\bar{t}\cdot X_i=f_{\bar{a}_i}(\overline{t})X_i$ for all $\bar{t}\in\overline{T}$, 
where $f_{\bar{a}_i}:\overline{T}\to\LL^\times$ is the unique $1$-cocycle that restricts to $\bar{\beta}(\cdot,\bar{a}_i)$
on $\overline{K}$ and satisfies $f_{\bar{a}_i}(\bar{t}_0(C))=\lambda_i^{\frac{1}{[\LL^\times|/|\FF^\times|}}$.
Pick a primitive idempotent $E\in\overline{\cC}^\op$ and define a $\overline{T}$-graded $\FF$-algebra 
\[
\overline{\cD}\bydef E\bigl(\overline{\cC}^\op\# \FF\overline{T}\bigr)E=
\bigoplus_{\bar{t}\in\overline{T}}\overline{D}_{\bar{t}}
\text{ where }\overline{D}_{\bar{t}}E\bigl(\overline{\cC}^\op\# \bar{t}\bigr)E.
\]
Finally, for any character $\chi:K_{[|\FF^\times|]}\to\FF^\times$, pick an extension 
$\tilde{\chi}:T\to GF(p^\infty)^\times$ and define the following $G$-graded $\FF$-algebra with support $T$:
\[
\cD\bydef\bigoplus_{t\in T}\overline{\cD}_{\pi(t)}\otimes\tilde{\chi}(t)^{\frac{1}{|\FF^\times|}}t
\subset\overline{\cD}\otimes_\FF GF(p^\infty)T,
\]
where $\pi:T\to\overline{T}$ is the natural homomorphism. 
This is a graded-central-division algebra over $\FF$, whose graded-isomorphism class does not depend on the choice of $E$
or $\tilde{\chi}$, but depends on our (fixed) choice of the generators $\bar{a}_i$ and coset representative $\bar{t}_0(C)$.
By abuse of notation, we will denote this object by $\cD(\FF,T,H,C,\cO,\chi)$ if $\bar{\beta}=\bar{\beta}_0(\cO)$.
\end{df}

We can now summarize our classification:

\begin{theorem}\label{th:finite_GDR}
Let $G$ be an abelian group. Every finite $G$-graded-division ring is graded-isomorphic to some $\cD(\FF,T,H,C,\cO,\chi)$
as in Definition \ref{df:modelsD_finite_field}. Moreover, $\cD(\FF,T,H,C,\cO,\chi)$ and $\cD(\FF',T',H',C',\cO',\chi')$ are 
graded-isomorphic if and only if $\FF'=\FF$, $T'=T$, $H'=H$, $C'=C$, $\cO'=\cO$, and $\chi'=\psi\circ\chi$ for some 
$\psi\in\Gal(\FF/\FF_0)$, where $\FF_0$ is the subfield of $\FF$ generated by the values of the bicharacters in $\cO$. \qed
\end{theorem}



\begin{thebibliography}{KMRT98}

\bibitem[ABFP08]{Allison_et_al}
B.~Allison, S.~Berman, J.~Faulkner, and A.~Pianzola, \emph{Realization of 
graded-simple algebras as loop algebras}, Forum Math. \textbf{20} (2008), no.~3, 395--432.

\bibitem[AG60]{AuslanderGoldman}
M.~Auslander and O.~Goldman, \emph{The Brauer group of a commutative ring}, 
Trans. Amer. Math. Soc. \textbf{97} (1960), 367--409.

\bibitem[BEKpr]{BEK}
Y.~Bahturin, A.~Elduque, and M.~Kochetov, \emph{Graded-division algebras over arbitrary fields}, to appear in 
J.~Algebra Appl. (preprint arXiv:1912.11911).

\bibitem[BK19]{BK} Y.~Bahturin and M.~Kochetov, \emph{On nonassociative graded-simple algebras over the field of 
real numbers}, Tensor categories and Hopf algebras, 25--48, Contemp. Math., 728, Amer. Math. Soc., Providence, RI, 2019.

\bibitem[BSZ01]{BSZ}
Y.~Bahturin, S.~Sehgal, and M.~Zaicev, \emph{Group gradings on associative algebras},
J.~Algebra \textbf{241} (2001),  677--698.

\bibitem[BZ02]{BZ}
Y.~Bahturin and M.~Zaicev, \emph{Group gradings on matrix algebras}, Canad. Math. Bull. \textbf{45} (2002), 499--508.

\bibitem[BZ16]{BZreal_simple}
Y.~Bahturin and M.~Zaicev, \emph{Simple graded division algebras over the field of real numbers}, 
Linear Algebra Appl. \textbf{490} (2016), 102--123.

\bibitem[BZ18]{BZreal}
Y.~Bahturin and M.~Zaicev, \emph{Graded division algebras over the field of real numbers}, 
J.~Algebra \textbf{514} (2018), 273--309.

\bibitem[Cae98]{CaeBr} S.~Caenepeel, \emph{Brauer groups, Hopf algebras and Galois theory}, 
K-Monographs in Mathematics, 4, Kluwer Academic Publishers, Dordrecht, 1998. 

\bibitem[CR66]{ChaseRosenberg}
S.U.~Chase and A.~Rosenberg, \emph{A theorem of Harrison, Kummer Theory, and Galois algebras}, 
Nagoya Math.~J. \textbf{27} (1966), 663--685.

\bibitem[CGO73]{CGO}
L.N.~Childs, G.~Garfinkel, and M.~Orzech, \emph{The Brauer group of graded Azumaya algebras}, 
Trans.~Amer.~Math.~Soc. \textbf{175} (1973), 299--326.

\bibitem[Dav01]{Dav}
A.A.~Davydov, \emph{Galois algebras and monoidal functors between categories of representations of finite groups}, 
J.~ Algebra \textbf{244} (2001), 273--301.

\bibitem[Eld19]{EldLoop}
A.~Elduque, \emph{Graded-simple algebras and cocycle twisted loop algebras}, 
Proc. Amer. Math. Soc. \textbf{147} (2019), no.~7, 2821--2833. 

\bibitem[EK13]{EKmon}
A.~Elduque and M.~Kochetov, \emph{Gradings on simple Lie algebras},
Mathematical Surveys and Monographs, 189, American Mathematical Society, Providence, RI; 
Atlantic Association for Research in the Mathematical Sciences (AARMS), Halifax, NS, 2013.

\bibitem[EK15a]{EK_Israel}
A.~Elduque and M.~Kochetov, \emph{Graded modules over classical simple Lie algebras with a grading}, 
Israel J.~Math. \textbf{207} (2015), no.~1, 229--280.

\bibitem[EK15b]{EK15}
A.~Elduque and M.~Kochetov, \emph{Gradings on the Lie algebra $D_4$ revisited} 
J.~Algebra \textbf{441} (2015), 441--474.

\bibitem[GN08]{GN}
C.~Galindo and S.~Natale, \emph{Normal Hopf subalgebras in cocycle deformations of finite groups},
Manuscripta Math. \textbf{125} (2008), 501--514.

\bibitem[Her68]{Her}
I.N.~Herstein, \emph{Noncommutative rings}, reprint of the 1968 original, 
Carus Mathematical Monographs, 15, Mathematical Association of America, Washington, DC, 1994. 

\bibitem[Kar73]{Kar}
G.~Karrer, \emph{Graded division algebras}, Math. Z. \textbf{133} (1973), 67--73.

\bibitem[Knu69]{Knus}
M.-A.~Knus, \emph{Algebras graded by a group}, Category Theory, Homology Theory and their Applications, II 
(Battelle Institute Conference, Seattle, Wash., 1968, Vol. Two), 117--133, Springer, Berlin, 1969.

\bibitem[KMRT98]{KMRT}
M.-A.~Knus, A.~Merkurjev, M.~Rost, and J.-P.~Tignol, \emph{The book of involutions}, 
American Mathematical Society Colloquium Publications, 44, American Mathematical Society, 
Providence, RI, 1998. 

\bibitem[MacL67]{MacLane}
S.~MacLane, \emph{Homology}, Die Grundlehren der mathematischen Wissenschaften, Band 114, 
Springer-Verlag, Berlin-New York, 1967.

\bibitem[Mon93]{Mont} Montgomery, S. {\it Hopf algebras and their actions on rings}, 
CBMS Regional Conference Series in Mathematics, 82, Amer. Math. Soc., Providence, RI, 1993.

\bibitem[Mov93]{Mov}
M.V.~Movshev, \emph{Twisting in group algebras of finite groups} (Russian), 
Funktsional. Anal. i Prilozhen. \textbf{27} (1993), no.~4, 17--23, 95; 
translation in Funct. Anal. Appl. \textbf{27} (1993), no.~4, 240--244 (1994).

\bibitem[NvO04]{NvO}
C.~Nastasescu and F.~Van Oystaeyen, \emph{Methods of Graded Rings},
Lecture Notes in Mathematics, 1836, Springer-Verlag, Berlin, 2004.

\bibitem[PP70]{PP}
D.J.~Picco and M.I.~Platzek, \emph{Graded algebras and Galois extensions}, Rev. Un. Mat. Argentina 
\textbf{25}, (1970/71), 401--415.

\bibitem[RE16]{ARE}
A.~Rodrigo-Escudero, \emph{Classification of division gradings on finite-dimensional simple real
algebras}, Linear Algebra Appl. \textbf{493} (2016), 164--182.

\bibitem[Sch85]{Scharlau}
W.~Scharlau, \emph{Quadratic and Hermitian Forms}, Grundlehren der Mathematischen Wissenschaften, 270, 
Springer-Verlag, Berlin, 1985.

\bibitem[Swe69]{Sweedler}
M.E.~Sweedler, \emph{Hopf algebras},
Mathematics Lecture Note Series W. A. Benjamin, Inc., New York, 1969.

\bibitem[Wal63]{Wall}
C.T.C.~Wall, \emph{Graded Brauer Groups}, J.~Reine Angew. Math. \textbf{213} (1963/64), 187--199.


\end{thebibliography}
\end{document}